\documentclass[oneside,10pt]{amsart}          
\usepackage[b5paper]{geometry}	    
\usepackage{amsfonts,amsmath,latexsym,amssymb} 

\usepackage{amsthm} 

\usepackage{mathrsfs,upref} 
\usepackage{mathptmx} 
\usepackage{graphicx}                      
\usepackage{amsmath,amssymb,amsthm,amsfonts,mathtools}
\usepackage{enumerate}

\newcommand{\color}[2][{}]{}        

\usepackage{dsfont}   

\usepackage{bbm}        

\usepackage{mathrsfs}    
\renewcommand\mathcal\mathscr

\newcommand{\wt}{\widetilde}           

\newcommand{\mc}{\mathscr}
\newcommand{\ul}{\underline}

\numberwithin{equation}{section}

\theoremstyle{plain}            
\newtheorem{theorem}{Theorem}[section]

\newtheorem{proposition}[theorem]{Proposition}
\newtheorem{lemma}[theorem]{Lemma}
\newtheorem{corollary}[theorem]{Corollary}

\theoremstyle{definition}       
\newtheorem{definition}[theorem]{Definition}

\newtheorem{example}[theorem]{Example}

\theoremstyle{remark}
\newtheorem{remark}[theorem]{Remark}

\newcommand{\Sec}[1]{Section~\ref{sec:#1}}

\newcommand{\Fig}[1]{Figure~\ref{fig:#1}}

\newcommand{\Thm}[1]{Theorem~\ref{thm:#1}}

\newcommand{\Ex}[1]{Example~\ref{ex:#1}}

\newcommand{\Lem}[1]{Lemma~\ref{lem:#1}}

\newcommand{\Cor}[1]{Corollary~\ref{cor:#1}}

\newcommand{\Prp}[1]{Proposition~\ref{prp:#1}}

\newcommand{\Rem}[1]{Remark~\ref{rem:#1}}

\newcommand{\Def}[1]{Definition~\ref{def:#1}}

\newcommand{\R}{\mathbb{R}} 
\newcommand{\C}{\mathbb{C}} 
\newcommand{\N}{\mathbb{N}} 

\newcommand{\vx} {\mathsf{v}}  
\newcommand{\wx} {\mathsf{w}}  
\newcommand{\vxwx} {{\mathsf{vw}}}
\newcommand{\Vx} {\mathsf{V}}
\newcommand{\edge} {\mathsf{e}}

\newcommand{\Ed} {\mathsf{E}}

\newcommand{\vxeps}{{\eps,\vx}}
\newcommand{\Vxeps}{{\eps,\Vx}}
\newcommand{\wxeps}{{\eps,\wx}}
\newcommand{\vxwxeps}{{\eps,\vx\wx}}
\newcommand{\edeps}{{\eps,\edge}}
\newcommand{\Edeps}{{\eps,\Ed}}
\newcommand{\Edepsvx}{{\eps,\Ed_\vx}}

\newcommand{\vxed}{{\vx,\edge}}

\newcommand{\Vxnull}{{0,\Vx}}

\newcommand{\Ednull}{{0,\Ed}}

\newcommand{\e}{\mathrm e}  
\newcommand{\im}{\mathrm i} 

\newcommand{\eps}{\varepsilon}       
\renewcommand{\epsilon}{\varepsilon} 
\renewcommand{\phi}{\varphi}         

\DeclareMathOperator{\dd}    {d\!}      
\DeclareMathOperator{\dist}   {dist}
\DeclareMathOperator{\dom}    {dom}

\DeclareMathOperator{\id}     {id}  

\DeclareMathOperator{\volume}    {vol}
\DeclareMathOperator{\tr}     {tr}  

\DeclareMathOperator{\range}{Rg}
\DeclareMathOperator{\rank}{rk}

\renewcommand{\Re}     {\operatorname{Re}}
\renewcommand{\Im}     {\operatorname{Im}}

\newcommand{\de} {\mathord{\mathrm d}} 

\newcommand{\specsymb}{\sigma} 
\newcommand{\ressymb}{\varrho}

\newcommand{\spec}[2][{}]   {\specsymb_{\mathrm{#1}}(#2)}

\newcommand{\res}[1]{\ressymb(#1)}

 \newcommand{\Err}{\mathrm O}

\def\Xint#1{\mathchoice
   {\XXint\displaystyle\textstyle{#1}}%
   {\XXint\textstyle\scriptstyle{#1}}%
   {\XXint\scriptstyle\scriptscriptstyle{#1}}%
   {\XXint\scriptscriptstyle\scriptscriptstyle{#1}}%
   \!\int}
\def\XXint#1#2#3{{\setbox0=\hbox{$#1{#2#3}{\int}$}
     \vcenter{\hbox{$#2#3$}}\kern-.5\wd0}}

\newcommand{\dashint}{\Xint-}   
\newcommand{\avint}{{\textstyle\dashint}}   

\newcommand{\HS}{ H}           

\newcommand{\HSone}{ V}           

\newcommand{\Bsymb} {\mathscr L}       

\newcommand{\Sobsymb} {H}      
\newcommand{\Contsymb} {C}     
\newcommand{\Lsymb}    {L}     
\newcommand{\lsymb}    {\ell}          
\newcommand{\Lsqrsymb}    {\Lsymb^2}     
\newcommand{\lsqrsymb}    {\lsymb^2}          

\newcommand{\Lin}[1]{\mathscr \Bsymb({#1})}

\newcommand{\Cont}[2][{}]{\Contsymb^{#1}({#2})} 

\newcommand{\Lsqr}[1]{\Lsqrsymb({#1})} 
\newcommand{\lsqr}[1]{\lsqrsymb({#1})}   
\newcommand{\Linfty}[2][{}]{\Lsymb_\infty^{#1}({#2})} 

\newcommand{\Sob}[2][1]{\Sobsymb^{#1}({#2})}         

\newcommand{\abs}[1]{\lvert#1\rvert}   
\newcommand{\bigabs}[1]{\bigl\lvert#1\bigr\rvert}   
\newcommand{\Bigabs}[1]{\Bigl\lvert#1\Bigr\rvert}   
\newcommand{\abssqr}[2][{}]{\lvert{#2}\rvert^2_{#1}} 
\newcommand{\norm}[2][{}]{\|{#2}\|_{{#1}}}    
\newcommand{\normsqr}[2][{}]{\|{#2}\|^2_{#1}} 
\newcommand{\Bignorm}[2][{}]{\Bigl\|{#2}\Bigr\|_{#1}}     

\newcommand{\iprod}[3][{}]{\langle{#2}|{#3}\rangle_{#1}}  
\newcommand{\bigiprod}[3][{}]{\bigl\langle{#2}\bigl|\bigr.%
        {#3}\bigr\rangle_{#1}}

\newcommand{\set}[2]{\{ \, #1 \, : \, #2 \, \} } 
\newcommand{\bigset}[2]{\bigl\{ \, #1 \, : \, #2 \, \bigr\} }
\newcommand{\Bigset}[2]{\Bigl\{ \, #1 \, : \, #2 \, \Bigr\} }

\newcommand{\map}[3]{ #1 \colon #2 \longrightarrow #3}    

\newcommand{\bd}  {\partial}                

\newcommand{\dcup}{\mathbin{\mathaccent\cdot\cup}}

\DeclareMathOperator*{\bigdcup}{\mathaccent\cdot{\bigcup}}

\newcommand{\disjcup}{\mathrel{\dcup}} 
\newcommand{\bigdisjcup}{\operatorname*{\bigdcup}}

\newcommand{\restr}[1]{{\restriction}_{#1}} 

\newcommand{\conj}[1]{\overline {{#1}}}       
\newcommand{\normder}{\partial_\mathrm{n}}  

\newcommand{\1}{\mathbbm 1}            
\newcommand{\quadtext}[1]{\quad\text{#1}\quad}
\newcommand{\qquadtext}[1]{\qquad\text{#1}\qquad}

\newcommand{\Neu}{{\mathrm N}}              

\newcommand{\cvol}[1]{c_{\volume,#1}}   
\newcommand{\ctr}[1]{C^{\tr}_{#1}}   

\newcommand{\Jup}{J^{\uparrow\eps}}
\newcommand{\Jdown}{J^{\downarrow\eps}}
\newcommand{\Jnup}{J^{\uparrow n}}
\newcommand{\Jndown}{J^{\downarrow n}}

\begin{document}


\title[Norm convergence on varying Hilbert spaces]{Norm convergence of
  sectorial operators on varying Hilbert spaces}

\author{Delio Mugnolo}
\address{Institut f\"ur Analysis, Universit\"at
  Ulm, Helmholtzstra\ss e~18, 89069 Ulm, Germany}
  \email{delio.mugnolo@uni-ulm.de}


\author{Robin Nittka}

\address{Institut f\"ur Angewandte Analysis, Universit\"at
  Ulm, Helmholtzstra\ss e~18, 89069 Ulm, Germany}
  \email{robin nittka@uni-ulm.de}
  

\author{Olaf Post}
\address{School of Mathematics, Cardiff University, Senghennydd Road,
  Cardiff CF24 4AG, UK (on leave from: Department of Mathematical
  Sciences, Durham University, UK)}
   \email{olaf.post@durham.ac.uk}




\begin{abstract}
  Convergence of operators acting on a given Hilbert space is an old
  and well studied topic in operator theory.  The idea of introducing
  a related notion for operators acting on varying spaces is natural.
  Many previous contributions to this subject consider either concrete
  examples of perturbations, or an abstract setting where weak or
  strong convergence of the resolvents is used.  However, it seems
  that the first results on \emph{norm} resolvent convergence in this
  direction have been obtained only recently, to the best of our
  knowledge.  Here we consider sectorial operators on Hilbert spaces
  that depend on a parameter.  We define a notion of convergence that
  generalises convergence of the resolvents in operator norm to the
  case when the operators act on different spaces.  In addition, we
  show that this kind of convergence is compatible with the functional
  calculus of the operator and moreover implies convergence of the
  spectrum.  Finally, we present examples for which this convergence
  can be checked, including convergence of coefficients of parabolic
  problems.  Convergence of a manifold (roughly speaking consisting of
  thin tubes) towards the manifold's skeleton graph plays a prominent
  role, being our main application.
\end{abstract}

\maketitle

\section{Introduction}
\label{sec:intro}

Convergence of operators in the resolvent sense is a classical issue
in operator theory. Early results go back, at least implicitly, to
Rayleigh and Schr\"odinger. The first systematic investigations are
due to Trotter, Rellich and Kato.

If the operators under consideration arise from sesquilinear forms on
a Hilbert space, there are powerful methods available to study
convergence of the operators, in particular in the self-adjoint case.
In Kato's classical monograph~\cite{Kato95} 
one finds a detailed study of various kinds of convergence with focus
on strong and norm convergence in the resolvent sense and the
consequences of the respective convergence for the behaviour of the
spectrum.  Moreover, Kato gives criteria in terms of the forms that
allow to check easily in many situations that a sequence of operators
arising from uniformly sectorial forms converge either strongly or in
norm.  Those criteria are particularly easy to verify if the forms
satisfy some monotonicity assumptions, i.e., they converge from above
or from below.

A similar, very successful approach has been developed by
Mosco~\cite{Mos67} in the context of symmetric Dirichlet forms, i.e.,
forms associated with sub-Markovian self-adjoint
$\mathrm{C}_0$-semigroups.  He succeeds in obtaining \emph{strong}
resolvent convergence, spectral convergence and convergence of the
generated semigroups from simple conditions on the forms, and in fact
resolvent convergence can be easily characterised via the forms.

In the context of homogenisation problems, convergence results for
operators acting in different spaces have been considered e.g.\ in
\cite{sanchez-palencia:80, osy:92, koz:94, melnyk:01, melnyk:03,
  pastukhova-zhikov:07} on an abstract level and in concrete examples
like $\Lsqr{\Omega,\mu_\eps}$ with $\eps$-depending measures
$\mu_\eps$ converging weakly to a measure $\mu_0$, or even with
changing domains.  In the case of homogenisation problems on varying
domains, the identification operators often consists of restrictions
and extensions of functions, and the latter operator is not always
bounded on the form domains~(see e.g.~\cite{melnyk:00}).  Note that
these results imply strong or weak convergence of the resolvents, and
imply convergence of the discrete spectrum.  Moreover, the strong
convergence of the corresponding semigroups follows,
see~\cite{pastukhova-zhikov:07} and references therein.  On the other
hand, these methods can also be extended to certain nonlinear
settings, cf.~\cite{melnyk:07}.

On the other hand, a natural approach to infinite dimensional problems
is based on approximation via finite dimensional spaces, see
e.g.~\cite{ito-kap:02}. If in particular one considers diffusion-like
processes, form methods are a mighty tool.  Convergence schemes for
Dirichlet forms on varying spaces of finite dimension have been
considered by Mosco and others, particularly in the context of
stochastic diffusion equations and diffusion on fractals, see
e.g.~\cite{kolesnikov:06,freiberg:08,hinz:09,andres:09,mosco:09}.
There are similar convergence results for manifolds, metric measure
spaces, Hilbert spaces, quadratic forms on different Hilbert spaces
in~\cite{kuwae-shioya:03,kasue:02,kasue:06}.  Though, in these works
only the \emph{strong} convergence of the associated operators is
considered.

Moreover, elliptic equations on varying domains with respect to
several boundary conditions on several spaces have been widely
studied.  We refer to the work of Stollmann~\cite{stollmann:95} on
strong and norm resolvent convergence of Dirichlet Laplacians on
varying domains (see also \cite{weidmann:84} for the strong resolvent
convergence), as well as the works of Stolz and Weidmann about the
approximation of singular Sturm-Liouville operators by regular ones;
using again a domain change, see e.g.~\cite{stolz-weidmann:95}.
Finally, we refer to Daners' survey article~\cite{Dan08} for more
results on problems in the spirit of form methods and further
references.

In this article, in contrast, we are interested in convergence
properties \emph{in operator norm} of operators associated with forms
that act on varying Hilbert spaces, where the identification operators
are not necessarily given in a canonical way.  Although our setting is
more restrictive than e.g.\ the strong convergence one used in
homogenisation problems, there is still a wide class of examples in
which the necessary assumptions are naturally fulfilled.  We would
like to stress that the convergence in operator norm of the resolvents
in~\cite{stollmann:95} uses the fact that all spaces are canonically
embedded in a common space $\Lsqr{\R^d}$, which is not necessarily
true in our situation.

Let us now describe the results of this article in more details.  We
investigate convergence of m-sectorial operators $A_\eps$, which are
allowed to act on different Hilbert spaces $H_\eps$, towards an
m-sectorial operator $A_0$ acting on a Hilbert space $H_0$ by form
methods. Our notion of form convergence resembles a sufficient
condition for convergence of the resolvent in operator norm due to
Kato and is designed in a way that allows to check the conditions
easily in many applications.  The notation is introduced in
\Sec{general}.  Our main abstract results are contained in
\Sec{abstract}.  More precisely, in \Sec{functional_calc} we show that
if $A_\eps$ converges to $A_0$, then also $\phi(A_\eps)$ converges to
$\phi(A_0)$ in norm if $\phi$ is in a suitable class of bounded
holomorphic functions (\Thm{hol.calc2}). We prove in
\Sec{spectral_conv} that the spectra of $A_\eps$ converge to the
spectrum of $A_0$ (\Cor{conn_spec} and \Thm{eigenvalue_convergence}).
Similar results for self-adjoint operators can be found
in~\cite{post:06}.  In~\cite{post:pre11}, also convergence of certain
non-self-adjoint operators in a specific situation is considered. In
\Sec{extra} we consider invariance of subsets of the Hilbert spaces
and extrapolated semigroups.  In particular, if we assume that the
Hilbert spaces $H_\eps$ are $L^2$-spaces and the semigroups
$(\e^{tA_\eps})_{t \ge 0}$ generated by the $A_\eps$ are bounded on
the corresponding $L^\infty$-spaces, then we can prove that under
suitable assumptions on the convergence scheme the semigroups
$\e^{tA_\eps}$ converge to $\e^{tA_0}$ also as operators on $L^p$ for
$p \in [2,\infty)$ (\Thm{Lpconv}).

\Sec{simple_example} describes several situations to which our results
can be applied without much effort.  In \Sec{fourier} we put the
Fourier series expansion with respect to eigenvectors into our
framework to exhibit the ideas at an elementary example.  In
\Sec{coeff_conv} we apply our results in a situation where $A_\eps$
are elliptic operators on a domain whose coefficients converge to the
coefficients of an elliptic operator $A_0$. More precisely, we
consider generalised Wentzell-Robin boundary conditions, which are a
natural candidate for our framework because the natural choice of
inner products on the underlying Hilbert space depends on the
coefficients even if the Hilbert spaces coincide as sets.  In this
setting we generalise results of Coclite et al.~\cite{coc-fav:08} and
complement those of~\cite{coc-gol:08} (\Thm{var.coeff}).  In
\Sec{degenerate} we adopt a variational approach to elliptic operators
whose coefficient may vanish at the boundary, as in~\cite{AC08}
(\Thm{deg.eq}). Observe that in this situation the limiting Hilbert
space differs from the approximating ones --- not only with respect to
the inner product, but even as a set ---, so that Kato's classical
results cannot be applied directly.

Our main example, however, is the convergence of tube-like manifolds
to the skeleton graph, which we investigate in
Section~\ref{sec:robin}.  More precisely, we let $H_\eps =
\Lsqr{X_\eps}$ for $\eps>0$, where $X_\eps$ is a manifold consisting
of $(m+1)$-dimensional objects resembling tubes (\emph{edge
  neighbourhoods}) that are connected in $(m+1)$-dimensional junction
regions (\emph{vertex neighbourhoods}). If these tubes have a uniform
thickness $\eps$, then it is natural to expect that the behaviour of
physical processes on $X_\eps$ which are described by an elliptic
operator is close to the behaviour of an analogous process on the
\emph{skeleton graph} $X_0$, which is a $1$-dimensional manifold with
singularities at the vertices.  We show that under some uniformity
assumptions we indeed have resolvent convergence and convergence of
finite parts of the spectrum (\Thm{conv.mfd}).  Note that the
convergence results for manifolds and metric measure spaces of Kasue
et al.~\cite{kuwae-shioya:03,kasue:02,kasue:06} cannot be used here,
since our families of manifolds $(X_\eps)_\eps$ do not satisfy the
necessary curvature bounds (see e.g.~\cite{kasue:02}, p.~1224).

Robin boundary conditions are closely related to Neumann boundary
conditions from the perspective of the quadratic (or, more generally,
sesquilinear) form approach. In fact, the form domain is the same,
while the forms differs only by a (possibly non-symmetric) boundary
term. This allows us to rely on the results in~\cite{post:06} for
treating the principal term, so that we only have to handle the
boundary term.

One of our motivations for this example is given by the articles of
Grieser~\cite{grieser:pre07} and Cacciapuoti--Finco~\cite{CF08}.
Grieser considered general boundary conditions (Dirichlet, Robin or
Neumann) on a manifold (if embedded, the embedding is ``straight'')
shrinking to a metric graph.  He showed that the limit behaviour
depends on the scattering matrix at the threshold of the essential
spectrum, so that, generically, the limit operator is decoupling.
Cacciapuoti and Finco use a simple wave-guide model (in our
terminology, a flat manifold converging to a graph consisting of two
(half-infinite) edges and one vertex only).  Using curved embedded
edges with different scalings of the transversal and longitudinal
curvature, they obtain non-trivial couplings starting with Robin
boundary conditions.  However, their notion of convergence differs
significantly from ours since one can use separation of variables due
to the simple product topology of the space.  For the convergence of
unitary groups in a similar setting we refer to the work of Teufel and
Wachsmuth~\cite{teufel-wachsmuth:pre09}.

Grieser and Cacciapuoti-Finco use scale-invariant Robin boundary
conditions of the form $\frac{\partial u}{\partial \nu} = \beta_\eps
u$ with $\beta_\eps = \beta/\eps$. This scaling leads to transversal
eigenvalues of the order $\eps^{-2}$.  In particular, a rescaling of
the limit operator is necessary in order to expect convergence, see
\Rem{scaling}. Using Robin boundary
conditions with coupling of order $\beta_\eps= \Err(1)$ near the
vertices and $\beta_\eps=\Err(\eps^{3/2})$ along the edge
neighbourhoods, we are able to construct a family of manifolds with
boundary, such that, in the limit, the corresponding Laplacians
converge to a Laplacian on the underlying metric graph with
generalised, possibly non-local $\delta$-interactions in the vertices.
Using the same idea as in~\cite{exner-post:09}, we can further
approximate other couplings like the $\delta$'-interaction.

\subsection*{Acknowledgements}
This article has been written while the third author was visiting the
University of Ulm. He would like to thank the University of Ulm for
the hospitality and the financial support.

\section{Notation}
\label{sec:general}

We consider m-sectorial operators (in the sense of Kato) on Hilbert
spaces.  For our approach, it is convenient to work with such an
operator in terms of its associated form.  We briefly sketch the
correspondence of m-sectorial operators and sesquilinear forms.  For
these results and much more information we refer to~\cite{Kato95},
Chapter~VI.  We point out that there is a one-to-one correspondence
between bounded, $H$-elliptic forms and m-sectorial operators, so
there is no loss of generality in working with an m-sectorial operator
only in terms of its form.

Let $H$ be a Hilbert space and let $V$ be a dense subspace of $H$ that
is a Hilbert space in its own right, and which is continuously
embedded into $H$.  We say that a sesquilinear form $a\colon V \times
V \to \C$ is \emph{bounded} if there exists $M \ge 0$ such that
\begin{equation}\label{eq:form_bdd}
  \abs{a(u,v)} \le M \norm[V]{u} \norm[V]{v}
  \qquadtext{for all $u, v \in V$,}
\end{equation}
and we call $a$ \emph{$H$-elliptic} or simply \emph{elliptic} if there exist $\omega \in \R$
and $\alpha > 0$ such that
\begin{equation}\label{eq:form_coercive}
  \Re a(u,u) + \omega \normsqr[H]{u} \ge \alpha \normsqr[V]{u}
  \qquadtext{for all $u \in V$.}
\end{equation}
In this case
\[
	\norm[a]{u} \coloneqq \sqrt{ \Re a(u,u) + \omega \normsqr[H]{u} }
\]
defines an equivalent norm on $V$. More precisely, since $V$ is
continuously embedded into $H$, there exists $c \ge 0$ such that
\begin{equation}
  \label{eq:form_embedding}
  \norm[H]{u} \le c_V \norm[V]{u} \qquadtext{ for all $u \in V$.}
\end{equation}
For any such constant $c_V$, we obtain
\begin{equation}\label{eq:equiv_norm}
  \alpha \normsqr[V]{u} \le \normsqr[a]{u} \le \bigl(M + c_V^2
  \omega\bigr) \normsqr[V]{u} 
  \qquadtext{for all $u \in V$.}
\end{equation}

We define the \emph{associated operator} $A$ of $a$ by
\[
	u \in D(A) \text{ and } Au = f
	\quad :\Longleftrightarrow \quad
	u \in V \text{ and } a(u,v) = \iprod[H]{f}{v} \; \forall v \in V,
\]
and we emphasise that since the form $a$ is not assumed to be
symmetric, the associated operator $A$ is in general not self-adjoint.

Consider for a moment the form $b\colon V \times V \to \C$ given by
\[
	b(u,v) \coloneqq a(u,v) + \omega \iprod[H]{u}{v},
\]
which is associated with the operator $A+\omega$.
Then by~\eqref{eq:form_bdd} and~\eqref{eq:equiv_norm}
\[
	\abs{\Im b(u,u)}
		= \abs{\Im a(u,u)}
		\le \abs{a(u,u)}
		\le M \normsqr[V]{u}
		\le \frac{M}{\alpha} \normsqr[a]{u}
		= \frac{M}{\alpha} b(u,u).
\]
The proof of~Theorem~1.53 of \cite{Ouh05} now shows that
$\spec{A+\omega} \subset \overline \Sigma _{\arctan \frac M \alpha}$,
where
\begin{equation}
  \label{eq:def.sec}
    \Sigma_\theta := \bigset {z \in \C \setminus \{0\}} { \abs{\arg(z)} <
    \theta}.
\end{equation}
Moreover, denoting here and in the following
\[
R(z,A)\coloneqq (z - A)^{-1},
\]
\newcommand{\thetainterval}{(\arctan{\tfrac{M}{\alpha}}, \pi]}
for every $\theta \in \thetainterval$ we have
\[
	\norm[\mathscr{L}(H)]{zR(z,A+\omega)} \le D_\theta
	\quadtext{for all $z \not\in \Sigma_\theta$,} 
\]
i.e., $\spec{A} \subset -\omega + \Sigma_{\arctan\frac{M}{\alpha}}$ and
\begin{equation}\label{eq:resolvent_est}
	\norm[\mathscr{L}(H)]{R(z,A)} \le
	\frac{D_\theta}{\abs{z+\omega}}
	\quadtext{for all $z \not\in \Sigma_\theta - \omega$} 
\end{equation}
with
\[
	D_\theta \coloneqq \frac{1}{\sin(\theta - \arctan\tfrac{M}{\alpha})}.
\]
Operators satisfying such a condition are frequently called \emph{m-sectorial}
(in the sense of Kato).

\begin{definition}\label{def:equi_sec}
  Let $(H_\eps)_{\eps \ge 0}$ be a family of Hilbert spaces.  We say
  that $(a_\eps)_{\eps \ge 0}$ is an \emph{equi-elliptic family of
    sesquilinear forms with form domains $(V_\eps)_{\eps \ge 0}$}, if
  there exist $M$, $\omega$, $\alpha$ and $c_V$ not depending on
  $\eps$ such that~\eqref{eq:form_bdd}, \eqref{eq:form_coercive}
  and~\eqref{eq:form_embedding} are satisfied for all $\eps \ge 0$,
  i.e., all the constants are uniform with respect to $\eps$.  We call
  $\omega$ the associated \emph{vertex} and $\arctan (M/\alpha)$ the
  associated \emph{semi-angle}.
\end{definition}

\begin{remark}\label{rem:equiv_norm}
  If $(a_\eps)_{\eps \ge 0}$ is an equi-elliptic family of
  sesquilinear forms, then by~\eqref{eq:equiv_norm} the norms
  $\norm[V_\eps]{\cdot}$ and $\norm[a_\eps]{\cdot}$ are equivalent
  with a uniform constant. This allows us to use either of these two
  norm interchangeably in the following. For the theoretical part, the
  form norm is more convenient.  But for applications, we usually
  prefer to equip $V_\eps$ with other norms that are easier to handle.
\end{remark}

Now let $(a_\eps)_{\eps \ge 0}$ be a family of equi-elliptic family
of sesquilinear forms on Hilbert spaces $(H_\eps)_{\eps \ge 0}$.  We
want to ``measure'' the distance between the associated operators
$(A_\eps)_{\eps > 0}$ and $A_0$.  For this, we introduce
\emph{identification operators} $\map \Jup {H_0} {H_\eps}$ and $\map
\Jdown {H_\eps} {H_0}$, $\eps > 0$, which are considered to be
``almost unitary'', i.e., unitary up to some error.  For technical
reasons, it is also convenient to introduce identification operators
$\map {\Jup_1} {V_0} {V_\eps}$ and $\map{\Jdown_1} {V_\eps} {V_0}$ for
the form domains, which are considered to be ``almost the
restrictions'' of $\Jup$ and $\Jdown$ to $V_0$ and $V_\eps$,
respectively.

We make this more explicit and use the following terminology, inspired
by the technique developed in Appendix~A of~\cite{post:06}
and~\cite{post:pre11}.

\begin{definition}
  \label{def:quasi-uni}
  Let $\eps > 0$, and let $a_0$ and $a_\eps$ be bounded, elliptic,
  sesquilinear forms on Hilbert spaces $H_0$ and $H_\eps$ with form
  domains $V_0$ and $V_\eps$. Denote the associated operators by $A_0$
  and $A_\eps$, respectively.  For parameters $\delta_\eps > 0$ and
  $\kappa \ge 1$ we say that $a_0$ and $a_\eps$ are
  \emph{$\delta_\eps$-$\kappa$-quasi-unitarily equivalent} if there
  exist bounded operators $\Jup \in \mathscr{L}(H_0,H_\eps)$, $\Jdown
  \in \mathscr{L}(H_\eps,H_0)$, $\Jup_1 \in \mathscr{L}(V_0,V_\eps)$
  and $\Jdown_1 \in \mathscr{L}(V_\eps,V_0)$ that satisfy the
  following conditions.
  \begin{subequations}
    \label{eq:quasi-uni}
    \begin{gather}
      \label{A1}
      \norm[\mathscr{L}(V_0,H_\eps)]{\Jup - \Jup_1} \le \delta_\eps
      \quadtext{and} \norm[\mathscr{L}(V_\eps,H_0)]{\Jdown - \Jdown_1}
      \le
      \delta_\eps; \\
      \label{A2}
      \norm[\mathscr{L}(H_\eps,H_0)]{\Jdown - (\Jup)^\ast} \le
      \delta_\eps \quadtext{and}
      \norm[\mathscr{L}(H_0,H_\eps)]{\Jup - (\Jdown)^\ast} \le \delta_\eps; \\
	\label{A3}
        \norm[\mathscr{L}(V_0,H_0)]{\id - \Jdown \Jup} \le \delta_\eps
        \quadtext{and}
        \norm[\mathscr{L}(V_\eps,H_\eps)]{\id - \Jup \Jdown} \le
        \delta_\eps; \\
	\label{A4}
		\norm[\mathscr{L}(H_0,H_\eps)]{\Jup} \le \kappa
	\quadtext{and}
		\norm[\mathscr{L}(H_\eps,H_0)]{\Jdown} \le \kappa; \\
	\label{A5}
		\bigabs{a_0(f, \Jdown_1 u) - a_\eps(\Jup_1 f, u)} \leq \delta_\eps \norm[V_0]{f} \norm[V_\eps]{u}
		\quad\text{for all $f \in V_0$ and $u \in V_\eps$}.
    \end{gather}
  \end{subequations}
  If $(a_\eps)_{\eps \in \ge 0}$ is an equi-elliptic family of
  sesquilinear forms and if there exists $\kappa \ge 1$ and a family
  $(\delta_\eps)_{\eps > 0}$ of positive real numbers with $\lim_{\eps
    \to 0} \delta_\eps \to 0$ such that $a_\eps$ is
  $\delta_\eps$-$\kappa$-quasi-unitarily equivalent to $a_0$, then we
  say that the family $(a_\eps)_{\eps > 0}$ \emph{converges to $a_0$
    (in norm)}
  as $\eps \to 0$.
\end{definition}

\begin{remark}\label{rem:quasirem}\mbox{}
  \begin{enumerate}[(i)]
  \item For $\delta_\eps = 0$ the associated operators $A_0$ and
    $A_\eps$ are unitarily equivalent.  In fact, if $\delta_\eps = 0$,
    conditions~\eqref{A2} and~\eqref{A3} states that $\Jup$ is unitary
    with inverse $\Jdown$. Since by~\eqref{A1} the operators $\Jup_1$
    and $\Jdown_1$ are the restrictions of $\Jup$ and $\Jdown$,
    condition~\eqref{A5} states that $\Jup$ realises the unitary
    equivalence of $A_0$ and $A_\eps$.
  \item In the applications we have in mind, it is easy to check that
    $\map {\Jup_1}{V_0}{V_\eps}$ and $\map {\Jdown_1}{V_\eps}{V_0}$
    are bounded: if $\Jup_1$ is bounded as an operator into $H_\eps$
    and takes values in $V_\eps$, then it is bounded as an operator
    into $V_\eps$ by the closed graph theorem, and an analogous
    argument applies to $\Jdown_1$.  Note that in the context of
    homogenisation problems, the boundedness of $\Jdown_1$ is not
    always assured (see e.g.~\cite{melnyk:00}).

  \item The two conditions in~\eqref{A2} are equivalent to each other.
    In fact, they can be rephrased as
    \begin{equation}
      \label{eq:adjoint_scalar}
      \tag{\ref{A2}'}
      \abs{\iprod[H_\eps]{\Jup f}{u} - \iprod[H_0]{f}{\Jdown u} }
      \le \delta_\eps \norm[H_0]{f} \norm[H_\eps]{u}
      \text{ for all } f \in H_0 \text{ and } u \in H_\eps.
    \end{equation}
  \item Condition~\eqref{A3} does \emph{not} imply that $\Jdown \Jup$
    or $\Jup \Jdown$ are invertible operators. In fact, in most of our
    examples one of the two operators will have a large kernel,
    whereas the other has a small range.
  \item Only~\eqref{A5} depends on the evolution processes acting on
    $H_\eps$, while the first four conditions are solely related to
    the function spaces. So if we have verified~\eqref{A1}--\eqref{A4}
    in one situation, those conditions are satisfied for a large class
    of examples.  To be more specific, we can reuse the results
    obtained in~\cite{post:06} for Neumann boundary conditions and do
    not have to check these four conditions once again for the
    discussion of the Laplace operator with Robin boundary conditions
    in Section~\ref{sec:robin}.
  \end{enumerate}
\end{remark}

\begin{example}
  \label{ex:classical}
  Let $a_0$ and $a_\eps$ be forms on a single Hilbert space $H$ with
  equal form domain $V$ and let $\Jup$, $\Jup_1$, $\Jdown$ and
  $\Jdown_1$ be the identity on $H$ resp.\ $V$. Then the conditions of
  Definition~\ref{def:quasi-uni} are satisfied if and only if
  \begin{equation}\label{eq:form_Jid}
    \abs{a_0(u,v) - a_\eps(u,v)} \le \delta_\eps \norm[V]{u} \norm[V]{v}
  \end{equation}
  for all $u,v \in V$, i.e., $\norm{a_0-a_\eps} \to 0$ in the operator
  norm on the space of sesquilinear forms on $V$.
\end{example}

In the setting of \Ex{classical}, if~\eqref{eq:form_Jid} is satisfied
for a family $(\delta_\eps)_{\eps > 0}$ satisfying $\lim_{\eps \to 0}
\delta_\eps = 0$, then the resolvent of $A_\eps$ converges to the
resolvent of $A_0$ in operator norm uniformly on compact subsets of
$\res{A_0}$.  In fact, it would suffice if~\eqref{eq:form_Jid} is
satisfied whenever $u = v$, see Theorem~VI.3.6 of~\cite{Kato95}.  In
this sense, our results are a generalisation of this classical result
to the setting of varying spaces.  We can also deduce similar
consequences like in the classical situation, e.g.\ convergence of the
spectra.

\section{Abstract results}\label{sec:abstract}

For the whole section, let $(a_\eps)_{\eps \ge 0}$ be an
equi-elliptic family of sesquilinear forms for constants $M$,
$\omega$, $\alpha$ and $c_V$ as in~\eqref{eq:form_bdd},
\eqref{eq:form_coercive} and~\eqref{eq:form_embedding}, and let
$(A_\eps)_{\eps \ge 0}$ denote the associated operators.  We always
let the operators $\Jup$, $\Jdown$, $\Jup_1$ and $\Jdown_1$ and the
constant $\kappa$ be as in Definition~\ref{def:quasi-uni}.

\subsection{Functional calculus}
\label{sec:functional_calc}

In our situation, each operator $A_\eps + \omega$ is invertible by the
Lax-Milgram theorem due to~\eqref{eq:form_coercive}.  It is known that
in this situation the operators $A_\eps + \omega$ have bounded
$H^\infty$-calculus, see Sec.~11 of~\cite{kunwei:04}, Sec.~5.2
of~\cite{Are04} or Sec.~7.3 of~\cite{haase:06}.  We are going to show
that if $(a_\eps)_\eps$ converges to $a_0$ in the sense of
\Def{quasi-uni}, then the operators $\phi(A_\eps)$ converge to
$\phi(A_0)$ in a suitable sense as $\eps \to 0$ for an admissible
holomorphic function $\phi$.

We start with some auxiliary estimates. For brevity, in the proofs we
write
\[
R_\eps(z) \coloneqq R(z,A_\eps) = (z - A_\eps)^{-1}.
\]

\begin{lemma}
  \label{lem:rest}
  Let $\theta \in \thetainterval$.  There exists $C_\theta \ge 0$ such
  that for all $\eps \ge 0$ and $z \not\in \Sigma_\theta - \omega$
  \begin{equation*}
    \norm[\mathscr{L}(H_\eps,V_\eps)]{R(z,A_\eps)} \le
    \frac{C_\theta}{\sqrt{\abs{z+\omega}}}
    \quad\text{and}\quad
    \norm[\mathscr{L}(H_\eps,V_\eps)]{R(z,A_\eps)^\ast} \le
    \frac{C_\theta}{\sqrt{\abs{z+\omega}}}.
	\end{equation*}
\end{lemma}
\begin{proof}
  Let $u \in H_\eps$ be fixed.
  Then by~\eqref{eq:equiv_norm}
  \begin{align*}
    \alpha \normsqr[V_\eps]{R_\eps(z)u}
    & \le \normsqr[a_\eps]{R_\eps(z)u}
    = \Re a_\eps(R_\eps(z) u, R_\eps(z) u) + \omega \normsqr[H_\eps]{R_\eps(z)u} \\
    & = \Re \iprod[H_\eps]{(\omega+A_\eps) R_\eps(z)u}{R_\eps(z)u} \\
    & = \Re \iprod[H_\eps]{(\omega+z)R_\eps(z)u}{R_\eps(z)u} - \Re \iprod[H_\eps]{u}{R_\eps(z)u} \\
    & \le \bigl( \abs{\omega + z} \cdot \norm[\mathscr{L}(H_\eps)]{R_\eps(z)} + 1 \bigr) \norm[\mathscr{L}(H_\eps)]{R_\eps(z)} \normsqr[H_\eps]{u}
  \end{align*}
  Now~\eqref{eq:resolvent_est} implies the first estimate for
  \[
  C_\theta \coloneqq \sqrt{\frac{(1 + D_\theta) D_\theta}{\alpha}}.
  \]
  The second estimate can be proved like the first. In fact,
  $R(z,A_\eps)^\ast = R(\conj{z},A_\eps^\ast)$
  and $A_\eps^\ast$ is associated with the form $a_\eps^\ast$ given by
  \begin{equation*}
    a_\eps^\ast(u,v) \coloneqq \conj{a_\eps(v,u)}.
  \end{equation*}
	Thus it suffices to realise that $a_\eps^\ast$
	satisfies~\eqref{eq:form_bdd} and~\eqref{eq:form_coercive}
	for the same constants as $a_\eps$.
\end{proof}

\begin{lemma}\label{lem:clar}
	Let $\Jup$ and $\Jdown$ be as in \Def{quasi-uni}, and let
	$B_\eps \in \mathscr{L}(H_\eps,V_\eps)$ and
	$B_0 \in \mathscr{L}(H_0,V_0)$. Then
	\[
		\norm[\mathscr{L}(H_\eps)]{B_\eps - \Jup B_0 \Jdown}
			\le \kappa \norm[\mathscr{L}(H_\eps,H_0)]{\Jdown B_\eps - B_0 \Jdown}
				+ \delta_\eps \norm[\mathscr{L}(H_\eps,V_\eps)]{B_\eps}.
	\]
\end{lemma}
\begin{proof}
	By~\eqref{A3} and~\eqref{A4} we have
	\begin{align*}
		\norm[\mathscr{L}(H_\eps)]{B_\eps - \Jup B_0 \Jdown}
			& \le \norm[\mathscr{L}(H_\eps)]{B_\eps - \Jup \Jdown B_\eps}
				+ \norm[\mathscr{L}(H_\eps)]{\Jup \Jdown B_\eps - \Jup B_0 \Jdown} \\
			& \le \delta_\eps \norm[\mathscr{L}(H_\eps,V_\eps)]{B_\eps}
				+ \kappa \norm[\mathscr{L}(H_\eps,H_0)]{\Jdown B_\eps - B_0 \Jdown}.
	\qedhere
	\end{align*}
\end{proof}

\begin{lemma}\label{uuuugly}
	Let $\Jup$ and $\Jdown$ be as in \Def{quasi-uni},
	and let $f \in V_0$ and $u \in H_\eps$. Then
	\[
		\norm[H_0]{\Jdown u-f} \le
			\kappa \norm[H_\eps]{u - \Jup f}
				+ \delta_\eps \norm[V_0]{f}.
	\]
\end{lemma}

\begin{proof}
	Let $g\in H_0$. By \eqref{A3} and~\eqref{A4}
	\begin{align*}
		\bigabs{\iprod[H_0]{\Jdown u-f}{g}}
		& \le \bigabs{\iprod[H_\eps]{\Jdown (u - \Jup f)}{g}} 
                   + \bigabs{\iprod[H_\eps]{\Jdown \Jup f - f} g } \\
		& \le \kappa \norm[H_\eps]{u - \Jup f} \norm[H_0]{g} + \delta_\eps \norm[V_0]{f} \norm[H_0]{g}.
	\end{align*}
        	Since $g$ is arbitrary, this proves the claim.
\end{proof}

Now we prove the key estimate of this section.
\begin{proposition}
  \label{prp:updown}
  Let $(a_\eps)_\eps$ be an equi-elliptic family of sesquilinear
  forms with vertex $\omega$ and semi-angle $\theta_0:=\arctan
  (M/\alpha)$ as in \Def{equi_sec}, and let $a_\eps$ and $a_0$ be
  $\delta_\eps$-$\kappa$-quasi-unitarily equivalent (see
  \Def{quasi-uni}).  Moreover, let $r>0$ and $\theta \in (\theta_0,
  \pi]$.  Then there exist constants $C_{\theta,r,1} > 0$ and
  $C_{\theta,r,2} > 0$ such that
  \begin{equation}\label{downdown}
    \norm[\mathscr{L}(H_0,H_\eps)] {R(z,A_\eps) \Jup  - \Jup  R(z,A_0)}
    \le \frac{\delta_\eps C_{\theta,r,1}}{\sqrt{|z+\omega|}}
  \end{equation}
  and
  \begin{equation}\label{updown}
    \norm[\mathscr{L}(H_\eps)] {R(z,A_\eps) - \Jup  R(z,A_0) \Jdown}
    \le \frac{\delta_\eps C_{\theta,r,2}}{\sqrt{|z+\omega|}}
  \end{equation}
  for all $z \not\in \Sigma_\theta - \omega$ satisfying
  $\abs{z+\omega} \ge r$.
\end{proposition}

\begin{proof}
  Let $D_\theta$ be as in~\eqref{eq:resolvent_est} and let $C_\theta$
  be as in Lemma~\ref{lem:rest}.  Let $f \in H_0$ and $u \in H_\eps$
  be arbitrary, and fix $z \not\in \Sigma_\theta - \omega$.  Then
  by~\eqref{A2} (see also~\eqref{eq:adjoint_scalar}),
  \eqref{eq:resolvent_est}, \eqref{A1} and Lemma~\ref{lem:rest}
  \begin{align*}
    & \hspace*{-1em} \bigabs{\bigiprod[H_\eps]{\bigl(R_\eps(z) \Jup - \Jup R_0(z)\bigr)f} u } \\
    &\le \bigabs{\iprod[H_0]{f}{\Jdown R_\eps(z)^\ast u} -
      \iprod[H_\eps]{\Jup R_0(z) f}{u}}
    + \frac{\delta_\eps D_\theta}{|z+\omega|} \norm[H_\eps]{u} \norm[H_0]{f} \\
    &\le \bigabs{\iprod[H_0]{(z-A_0) R_0(z) f}{\Jdown_1 R_\eps(z)^\ast
        u}
      - \iprod[H_\eps]{\Jup_1 R_0(z) f}{(z-A_\eps)^\ast R_\eps(z)^\ast u}} \\
    & \qquad + \Bigl( \frac{2 \delta_\eps C_\theta}{|z+\omega|^{1/2}} + \frac{\delta_\eps D_\theta}{|z+\omega|} \Bigr) \norm[H_\eps]{u} \norm[H_0]{f}\\
    & \le \bigabs{ a_0(R_0(z)f, \Jdown_1 R_\eps(z)^\ast u) - a_\eps(\Jup_1 R_0(z) f, R_\eps(z)^\ast u) } \\
    & \qquad + |z| \bigabs{ \iprod[H_0]{R_0(z)f}{\Jdown_1 R_\eps(z)^\ast u} - \iprod[H_\eps]{\Jup_1 R_0(z)f}{R_\eps(z)^\ast u} } \\
    & \qquad + \Bigl( \frac{2 \delta_\eps C_\theta}{|z+\omega|^{1/2}}
    + \frac{\delta_\eps D_\theta}{|z+\omega|} \Bigr) \norm[H_\eps]{u}
    \norm[H_0]{f}.
  \end{align*}
  Using~\eqref{A5} and once again~\eqref{A2} we can further estimate
  \begin{align*}
    & \hspace*{-2em} \bigabs{\iprod[H_\eps]
      {\bigl(R_\eps(z) \Jup - \Jup R_0(z)\bigr)f} u }\\
    & \le \delta_\eps \norm[V_0]{R_0(z)f} \norm[V_\eps]{R_\eps(z)^\ast u} \\
    & \qquad + |z| \bigabs{ \iprod[H_0]{R_0(z)f}{\Jdown R_\eps(z)^\ast u} - \iprod[H_\eps]{\Jup R_0(z)f}{R_\eps(z)^\ast u} } \\
    & \qquad + \Bigl( \frac{\delta_\eps |z| C_\theta
      D_\theta}{|z+\omega|^{3/2}} + \frac{2 \delta_\eps
      C_\theta}{|z+\omega|^{1/2}} + \frac{\delta_\eps D_\theta}{|z+\omega|} \Bigr) \norm[H_\eps]{u} \norm[H_0]{f} \\
    & \le \Bigl( \frac{\delta_\eps (D_\theta +
      C_\theta^2)}{|z+\omega|} + \frac{\delta_\eps |z|
      D_\theta^2}{|z+\omega|^2} + \frac{\delta_\eps |z| C_\theta
      D_\theta}{|z+\omega|^{3/2}} + \frac{2 \delta_\eps
      C_\theta}{|z+\omega|^{1/2}}\Bigr) \norm[H_\eps]{u}
    \norm[H_0]{f}.
  \end{align*}
  For $\abs{z+\omega} \ge r$, this implies~\eqref{downdown} with
  \begin{equation*}
    C_{\theta,r,1}
    \coloneqq (C_\theta D_\theta + 2 C_\theta)
    + \frac{D_\theta + C_\theta^2 + D_\theta^2}{r^{1/2}} 
    + \frac{\abs \omega C_\theta D_\theta} r
    + \frac{\abs{\omega}  D_\theta^2} {r^{3/2}}.
  \end{equation*}
  Estimate~\eqref{updown} is a consequence of~\eqref{downdown} since
  by Lemma~\ref{lem:clar} and Lemma~\ref{lem:rest} we have
  \begin{equation*}
    \norm[\mathscr{L}(H_\eps)]{R_\eps(z)  - \Jup R_0(z) \Jdown}
    \le \kappa 
        \norm[\mathscr{L}(H_\eps,H_0)]{\Jdown R_\eps(z) - R_0(z)
          \Jdown}
    + \frac{\delta_\eps C_\theta}{\sqrt{|z+\omega|}},
  \end{equation*}
  so we can choose $C_{\theta,r,2} \coloneqq \kappa C_{\theta,r,1} +
  C_\theta$.
\end{proof}

\begin{remark}
  Estimate~\eqref{updown} tells us that we can find a good
  approximation of the operator $A_\eps$ in terms of the (often
  simpler) operators $A_0$, $\Jup$ and $\Jdown$, at least for small
  $\eps$. This is interesting by itself.  In fact, we even have a
  rather explicit error estimate; in the proof we have given concrete
  (though certainly not optimal) constants.  However, since these
  expressions are quite cumbersome, we prefer to work with the general
  constants $C_{\theta,r,1}$ and $C_{\theta,r,2}$.
\end{remark}

Define
\begin{equation*}
  H^\infty(\Sigma_\theta - \omega) 
  \coloneqq \bigset{\phi\colon \Sigma_\theta - \omega \to \C} {\phi \text{ is holomorphic and bounded}}
\end{equation*}
and
\begin{equation*}
  H^\infty_{00}(\Sigma_\theta - \omega)
  \coloneqq \bigset{\psi \in H^\infty(\Sigma_\theta - \omega)} 
             {\exists \mu > \tfrac{1}{2} \text{ such that } \; \psi(z) \in \Err(\abs{z}^{-\mu}) \; (z \to
    \infty)} 
\end{equation*}
and equip these spaces with the supremum norm.  Let $\theta \in
\thetainterval$.  We define the primary functional calculus of
$A_\eps$ for $\psi \in H^\infty_{00}(\Sigma_\theta - \omega)$ by
\begin{equation}\label{eq:prim_func_calc}
	\psi(A_\eps) \coloneqq \frac{1}{2\pi\im} \int_{\partial(\Sigma_\sigma - \omega)} \psi(z) R(z,A_\eps) \dd z,
\end{equation}
where $\sigma \in (\arctan{\tfrac{M}{\alpha}},\theta)$.  By Cauchy's
integral theorem, this definition is independent of the choice of
$\sigma$ and agrees with the usual definition of the functional
calculus, compare also Sec.~2.5.1 of~\cite{haase:06}.

\begin{remark}
  In our setting, the natural space for the primary functional
  calculus would be the larger space
  \begin{equation*}
    H^\infty_0(\Sigma_\theta)
    \coloneqq \bigset{\psi \in H^\infty(\Sigma_\theta)} 
        {\exists \; \mu > 0 \text{ such that }
           \psi(z) \in \Err(\abs{z}^{-\mu}) \; (z \to \infty)} 
  \end{equation*}    
  since in fact~\eqref{eq:prim_func_calc} is defined even for $\psi
  \in H^\infty_0(\Sigma_\theta - \omega)$.  However, using
  estimate~\eqref{updown} we can show convergence of $\psi(A_\eps)$ to
  $\psi(A_0)$ only for $\psi \in H^\infty_{00}(\Sigma_\theta -
  \omega)$.
\end{remark}

Since the operators $A_\eps$ are m-sectorial in the sense of Kato,
this functional calculus has a natural extension to $\phi \in H^\infty(\Sigma_\theta - \omega)$, and
the operator $\phi(A_\eps)$ is bounded with norm
\begin{equation}
  \label{eq:func_calc_bound}
  \norm[\Lin {H_\eps}]{\phi(A_\eps)}
  \le \Bigl( 2 + \frac 2 {\sqrt 3} \Bigr) \norm[\infty]{\phi},
\end{equation}
cf.~Corollary~7.1.17 of~\cite{haase:06}. It is important for us to
have a bound on the norm of $\phi(A_\eps)$ that is uniform with
respect to $\eps$.

We are now able to show that $\phi(A_\eps)$ converges to $\phi(A_0)$
is the following sense if $(a_\eps)_{\eps>0}$ converges to $a_0$ in
the sense of \Def{quasi-uni}, which is our main result in the context
of the functional calculus.

\begin{theorem}
  \label{thm:hol.calc2}
  Let $a_0$ and $(a_\eps)_\eps$ be equi-elliptic sesquilinear forms
  with associated operators $A_0$ and $(A_\eps)_\eps$ as in
  \Sec{general}.  Assume in addition that $a_\eps$ is
  $\delta_\eps$-$\kappa$-quasi-unitarily equivalent to $a_0$.  Let
  $\theta \in \thetainterval$.  Then for all $\psi \in
  H^\infty_{00}(\Sigma_\theta - \omega)$ there exists $C_\psi \ge 0$
  such that
	\begin{equation}\label{eq:psi_est}
		\norm[\mathscr{L}(H_\eps)]{\Jup\psi(A_0)\Jdown - \psi(A_\eps)} \le C_\psi \delta_\eps.
	\end{equation}
	Moreover, for all $\psi \in H^\infty(\Sigma_\theta - \omega)$ there exists $C_\psi \ge 0$ such that
	\begin{equation}\label{eq:phi_est}
		\norm[H_\eps]{\Jup \phi(A_0)\Jdown u-\phi(A_\eps)u}\le C_\phi \delta_\eps \norm[H_\eps]{(\omega+1+A_\eps) u}
	\end{equation}
	for all $u \in D(A_\eps)$.
\end{theorem}
\begin{proof}
	Fix $\psi \in H^\infty_{00}(\Sigma_\theta - \omega)$ and
	$\sigma \in (\arctan\tfrac{M}{\alpha},\theta)$.
	Let $\nu \in (0, \alpha c_V^{-2})$ be such that
	\[
		\theta' \coloneqq \arctan\Bigl(\frac{M}{\alpha - \nu c_V^2}\Bigr) < \sigma,
	\]
	where $\alpha$, $M$ and $c_V$ are as in~\eqref{eq:form_bdd}, \eqref{eq:form_coercive}
	and~\eqref{eq:form_embedding}. Since
	\[
		a_\eps(u,u) + (\omega - \nu) \normsqr[H_\eps]{u}
			\ge \alpha \normsqr[V_\eps]{u} - \nu \normsqr[H_\eps]{u}
			\ge (\alpha - \nu c_V^2) \normsqr[V_\eps]{u},
	\]
	the operator $A_\eps$ is m-sectorial with vertex $-\omega + \nu$ and semi-angle $\theta'$,
	and the same is true for $A_0$ on $H_0$.
	Hence by Proposition~\ref{prp:updown}
	\[
		\norm[\mathscr{L}(H_\eps)]{R(z,A_\eps) - \Jup R(z,A_0) \Jdown}
			\le \frac{\delta_\eps C_{\theta',\nu/2,2}}{\sqrt{\abs{z+\omega-\nu}}}
	\]
	for all $z \not\in \Sigma_{\theta'} - \omega + \nu$ such that $\abs{z+\omega-\nu} \ge \frac{\nu}{2}$.
	If $r > 0$ is sufficiently small, then $B(-\omega,r) \cap (\partial\Sigma_\sigma - \omega)$
	has distance at least $\frac{\nu}{2}$ to $\Sigma_{\theta'} - \omega + \nu$, and hence
	\begin{equation}\label{eq:close_est}
		\norm[\mathscr{L}(H_\eps)]{R(z,A_\eps) - \Jup R(z,A_0) \Jdown}
			\le \delta_\eps c_{\sigma,\nu}
	\end{equation}
	for all $z \in \partial(\Sigma_\sigma - \omega)$ satisfying $\abs{z+\omega} \le r$,
	where $c_{\sigma,\nu}$ and $r$ are constants depending on $\sigma$ and $\nu$,
	so in principle only on $M$, $\alpha$ and $c_V$.

	There exist $\mu > \frac{1}{2}$ and $K \ge 0$ such that
	\[
		\abs{\psi(z)} \le \frac{K}{\abs{z+\omega}^\mu}
		\quad\text{ and }\quad \abs{\psi(z)} \le K
	\]
	for all $z \in \Sigma_\theta - \omega$.
	Thus, by~\eqref{eq:prim_func_calc}, Proposition~\ref{prp:updown} and~\eqref{eq:close_est}
	\begin{align*}
		& \norm[\mathscr{L}(H_\eps)]{\Jup \psi(A_0) \Jdown - \psi(A_\eps)} \\
			& \qquad \le \frac{1}{2\pi} \int_{\partial(\Sigma_\sigma - \omega)} \abs{\psi(z)} \, \norm[\mathscr{L}(H_\eps)]{ \Jup R_0(z) \Jdown - R_\eps(z)} \dd z \\
			& \qquad \le \frac{1}{2\pi} \int_{\partial(\Sigma_\sigma - \omega) \setminus B(-\omega,r)} \frac{\delta_\eps C_{\sigma,r,2} K}{|z + \omega|^{\mu+\frac{1}{2}}} \dd z
				+ \frac{1}{2\pi} \int_{\partial(\Sigma_\sigma - \omega) \cap B(-\omega,r)} \delta_\eps c_{\sigma,\nu} K \dd z.
	\end{align*}
	Therefore, we have shown~\eqref{eq:psi_est} with
	\[
		C_\psi \coloneqq \frac{C_{\sigma,r,2} K}{(\mu - \frac{1}{2}) r^{\mu - \frac{1}{2}} \pi} + \frac{c_{\sigma,\nu} r K}{\pi}.
	\]

	In particular, there exists a constant $C_{\psi^\ast}$
    belonging to the function $\psi^\ast \in
    H^\infty_{00}(\Sigma_\theta)$ defined by
	\[
		\psi^\ast(z) \coloneqq \frac{1}{\omega + 1 + z}
	\]
	such that
	\[
		\norm[\mathscr{L}(H_\eps)]{\Jup \psi^\ast(A_0) \Jdown - \psi^\ast(A_\eps)} \le C_{\psi^\ast} \delta_\eps,
	\]
	i.e.,
	\begin{equation}\label{eq:res_est_func_calc}
		\norm[\mathscr{L}(H_\eps)]{\Jup (\omega + 1 + A_0)^{-1} \Jdown - (\omega + 1 + A_\eps)^{-1}} \le C_{\psi^\ast} \delta_\eps.
	\end{equation}

	Now let $\phi \in H^\infty(\Sigma_\theta - \omega)$ be fixed and define
	$\psi\in H^\infty_{00}(\Sigma_\theta - \omega)$ by
	\[
		\psi(z) \coloneqq \frac{\phi(z)}{\omega+1+z}
	\]
	Then $\phi(A_\eps) (\omega+1+A_\eps)^{-1} = \psi(A_\eps)$ by
        the construction of the functional calculus~ (see Sec.~2.3.2
        of~\cite{haase:06}.  Hence for all $u \in D(A_\eps)$ we have
	\begin{align*}
		& \norm[H_\eps]{\Jup \phi(A_0) (\omega + 1 + A_0)^{-1} \Jdown (\omega+1+A_\eps) u - \phi(A_\eps) u} \\
			& \qquad = \norm[H_\eps]{\Jup \psi(A_0) \Jdown (\omega+1+A_\eps)u - \psi(A_\eps) (\omega + 1 + A_\eps) u} \\
			& \qquad \le C_\psi \delta_\eps \norm[H_\eps]{(\omega + 1 + A_\eps)u}.
	\end{align*}
	Moreover, from~\eqref{A4} and \eqref{eq:func_calc_bound} we obtain that
	\begin{align*}
		& \norm[H_\eps]{\Jup \phi(A_0)\Jdown u - \Jup \phi(A_0) (\omega + 1 + A_0)^{-1} \Jdown (\omega + 1 + A_\eps) u}\\
			& \quad \le  \kappa \bigl( 2 + \frac{2}{\sqrt{3}} \bigr) \norm[\infty]{\phi}
				\norm[H_0]{\Jdown u - (\omega+1+A_0)^{-1} \Jdown (\omega+1+A_\eps) u}.
	\end{align*}
		
	Finally, by Lemma~\ref{uuuugly}, \eqref{eq:res_est_func_calc},
    Lemma~\ref{lem:rest} (for $z = -\omega-1$) and~\eqref{A4}
	\begin{align*}
		& \norm[H_0]{\Jdown u - (\omega+1+A_0)^{-1} \Jdown (\omega+1+A_\eps) u} \\
		& \quad \le \kappa \norm[H_\eps]{  u - \Jup (\omega+1+A_0)^{-1} \Jdown (\omega+1+A_\eps) u}\\
			& \quad\quad + \delta_\eps \norm[V_0]{ (\omega+1+A_0)^{-1} \Jdown (1+\omega+A_\eps) u} \\
		& \quad \le \kappa C_{\psi^\ast} \delta_\eps \norm[H_\eps]{(\omega+1+A_\eps) u}
			+ \delta_\eps C_\theta \kappa \norm[H_0]{(\omega+1+A_\eps)u}.
	\end{align*}
	Combining the previous three estimates, we have proved~\eqref{eq:phi_est} with
	\[
		C_\phi \coloneqq C_\psi + \kappa^2 
                \Bigl( 2 + \frac{2}{\sqrt{3}} \Bigr) \norm[\infty]{\phi}
                  \bigl( C_{\psi^\ast} + C_\theta \bigr).
	\qedhere
	\]
\end{proof}

\begin{corollary}
  \label{cor:hol.calc.cor}
  If $(a_\eps)_{\eps > 0}$ converges to $a_0$ in the sense of
  \Def{quasi-uni}, then the family $(\Jdown \phi(A_\eps) \Jup)_{\eps >
    0}$ converges in operator norm to $\phi(A_0)$ (regarded as
  operators on $H_0$) for every $\phi \in H^\infty_{00}(\Sigma_\theta
  - \omega)$, $\theta \in \thetainterval$.  If merely $\phi \in
  H^\infty(\Sigma_\theta - \omega)$, then we have at least convergence
  in the strong operator topology.
\end{corollary}
\begin{proof}
  Since all conditions in Definition~\ref{def:quasi-uni} are symmetric
  with respect to $a_0$ and $a_\eps$, interchanging the roles of the
  two forms we obtain as in Theorem~\ref{thm:hol.calc2} that
  \[
  \norm[\mathscr{L}(H_0)]{\psi(A_0) - \Jdown \psi(A_\eps) \Jup} \le
  C_\psi \delta_\eps
  \]
  for $\psi \in H^\infty_{00}(\Sigma_\theta - \omega)$, which proves
  the first claim.  Similarly,
  \[
  \norm[\mathscr{L}(H_0)]{\phi(A_0)f - \Jdown \phi(A_\eps) \Jup f} \le
  C_\phi \delta_\eps \norm[H_0]{(\omega+1+A_0) f}
  \]
  for all $f \in D(A_0)$ if $\phi \in H^\infty(\Sigma_\theta -
  \omega)$, implying that $(\Jdown \phi(A_\eps) \Jup)_{\eps > 0}$
  converges to $\phi(A_0)$ on a dense subspace of $H_0$.  Since the
  operators are uniformly bounded by~\eqref{A4}
  and~\eqref{eq:func_calc_bound}, this implies strong convergence.
\end{proof}

\begin{example}
  \label{ex:semigroups}
  The function $z \mapsto \e^{-tz}$ is in $H^\infty_{00}(\Sigma_\theta
  - \omega)$ for $\theta \in (0,\tfrac{\pi}{2}-\sigma)$ and for all $t
  \in \Sigma_\sigma$, $\sigma \in (0,\tfrac{\pi}{2})$.  Hence the
  semigroups $(\e^{-tA_\eps})_{\eps > 0}$ converge to $\e^{-tA_0}$ in
  operator norm (in the sense of \Thm{hol.calc2} and
  \Cor{hol.calc.cor}) for $t$ in the common sector of holomorphy of
  the semigroups, i.e., for every fixed $t \in \Sigma_\sigma$, where
  $\sigma \coloneqq \tfrac{\pi}{2} - \arctan\tfrac{M}{\alpha}$.

  Note, however, that we cannot expect this for $t=0$ since typically
  $\Jup \Jdown$ does not tend to the identity in operator norm even if
  $H_\eps = H_0$ for all $\eps \ge 0$. Thus we cannot expect uniform
  convergence near $t=0$. However, the explicit constant $C_\psi$ in
  the proof of Theorem~\ref{thm:hol.calc2} shows that the convergence
  is uniform on compact subsets of $\Sigma_\sigma$.
\end{example}

\subsection{Spectral convergence}\label{sec:spectral_conv}

It has been proven in Sec.~A.5 of~\cite{post:06} that if a family of
non-negative forms $(a_\eps)_{\eps > 0}$ converges to $a_0$ in the
sense of \Def{quasi-uni}, then the spectra of the associated operators
$\spec{A_\eps}$ converge to $\spec{A_0}$. In~\cite{post:pre11}, a
similar result for certain non-self-adjoint operators arising in the
treatment of resonances via complex scaling are considered.  Here we
prove that a similar result is true in the general (m-sectorial) case,
where the spectra need not to be real.  This is a part of the
justification why we regard our notion of convergence as a reasonable
generalisation of the classical resolvent convergence (see also
\Ex{classical}).

We consider the following notion of spectral convergence, which is
quite natural.  It is often called ``upper semi-continuity'' of the
spectrum.  This type of convergence is precisely what we obtain if in
a fixed Hilbert space we have a family of operators whose resolvents
converge in operator norm, see Theorem~IV.3.1 of~\cite{Kato95}.

\begin{definition}
  \label{def:spec_conv}
  We say that \emph{the spectra $\spec{A_\eps}$ of the family
    $(A_\eps)_{\eps > 0}$ converge to the spectrum $\sigma(A_0)$ of
    $A_0$ as $\eps\to 0$} if for each compact set $K \subset
  \res{A_0}$ there exists $\eps_1 > 0$ such that $K \subset
  \res{A_\eps}$ for all $\eps \in (0,\eps_1)$.
\end{definition}

Ideally, we could hope that the spectra $\spec{A_\eps}$ converge to
$\spec{A_0}$ if $(a_\eps)_{\eps > 0}$ converges to $a_0$.  In fact,
this is true if in addition $\res{A_0}$ is connected, see
Corollary~\ref{cor:conn_spec}.

We start with an auxiliary lemma, allowing us to estimate the resolvent of $A_\eps$
if we have a priori information about the resolvent of $A_0$.
For the whole section, the operators $(A_\eps)_{\eps > 0}$ are assumed
to satisfy the conditions in Section~\ref{sec:general}.

  \newcommand{\resconst}{\ell}
\begin{lemma}\label{lem:res_est_sector}
  For every $\resconst > 0$ and
  $r > 0$ there exist $\delta_0 = \delta_0(\resconst,r) > 0$ and $L =
  L(\resconst,r) > 0$ with the following property: if $a_\eps$ is
  $\delta_\eps$-$\kappa$-quasi-unitarily equivalent to $a_0$ for some
  $\delta_\eps \in (0,\delta_0]$, if $z \in \res{A_0} \cap
  \res{A_\eps} \cap B(0,r)$, and if $\norm[\mathscr{L}(H_0)]{R(z,A_0)}
  \le \resconst$, then $\norm[\mathscr{L}(H_\eps)]{R(z,A_\eps)} \le L$.
\end{lemma}
\begin{proof}
	For $z \in \res{A_0} \cap \res{A_\eps}$ we define
	\[
		V(z) \coloneqq \Jdown R_\eps(z) - R_0(z) \Jdown.
	\]
	Let $z$ and $z_0$ be in $\res{A_0} \cap \res{A_\eps}$.
	Then by the resolvent identity we have
	\[
		\bigl( R_0(z_0) - R_0(z) \bigr) \Jdown R_\eps(z) R_\eps(z_0) = R_0(z) R_0(z_0) \Jdown \bigl( R_\eps(z_0) - R_\eps(z) \bigr)
	\]
	and thus
	\[
		R_0(z_0) V(z) R_\eps(z_0) = R_0(z) V(z_0) R_\eps(z).
	\]
	Hence
	\begin{align*}
		V(z) & = (z_0 - A_0) R_0(z) V(z_0) R_\eps(z) (z_0 - A_\eps) \\
			& = (\id + (z_0 - z) R_0(z)) V(z_0) (\id + (z_0 - z) R_\eps(z))
	\end{align*}
	on $D(A_\eps)$ and thus on $H_\eps$ by density.
	Setting $z_0 \coloneqq -\omega - 1$ and using the dual
	version of~\eqref{downdown}, which follows by exchanging the roles
	of $A_\eps$ and $A_0$ in Proposition~\ref{prp:updown}, to estimate $V(z_0)$
	we deduce that
	\begin{equation}\label{eq:Vest}
		\norm[\mathscr{L}(H_\eps,H_0)]{V(z)}
			\le \delta_\eps C_{\theta,1,1} \bigl( 1 + \resconst \abs{\omega + 1 + z} \bigr) \bigl( 1 + \abs{\omega + 1 + z} \; \norm[\mathscr{L}(H_\eps)]{R_\eps(z)} \bigr).
	\end{equation}

	Next, we note that for all $u \in H_\eps$
	\begin{align*}
		\normsqr[a_\eps]{R_\eps(z)u}
			& = \iprod[H_\eps]{(\omega + A_\eps) R_\eps(z)u}{R_\eps(z)u} \\
			& \le \bigl( \norm[H_\eps]{u} + \abs{\omega + z} \; \norm[H_\eps]{R_\eps(z)u} \bigr) \norm[H_\eps]{R_\eps(z)u},
	\end{align*}
	proving by~\eqref{eq:equiv_norm} that
	\begin{equation}\label{eq:RHVest}
	\begin{aligned}
		\normsqr[\mathscr{L}(H_\eps,V_\eps)]{R_\eps(z)}
			& \le \frac{1}{\alpha} \bigl( 1 + \abs{\omega + z} \; \norm[\mathscr{L}(H_\eps)]{R_\eps(z)} \bigr) \norm[\mathscr{L}(H_\eps)]{R_\eps(z)} \\
			& \le \frac{1}{\alpha} \bigl( 1 + \beta \norm[\mathscr{L}(H_\eps)]{R_\eps(z)} \bigr)^2
	\end{aligned}
	\end{equation}
	with $\beta \coloneqq \max\{ 1, \abs{\omega + z} \}$.

	Now write
	\[
		R_\eps(z) = \bigl( \id - \Jup \Jdown \bigr) R_\eps(z) + \Jup \bigl( \Jdown R_\eps(z) - R_0(z) \Jdown \bigr) + \Jup R_0(z) \Jdown.
	\]
	This representation, combined with~\eqref{eq:Vest} and~\eqref{eq:RHVest}, shows that
	\begin{align*}
		\norm[\mathscr{L}(H_\eps)]{R_\eps(z)}
			& \le \delta_\eps \norm[\mathscr{L}(H_\eps,V_\eps)]{R_\eps(z)} + \kappa \norm[\mathscr{L}(H_\eps)]{V(z)} + \kappa^2 \resconst \\
			& \le \Bigl( \frac{\delta_\eps}{\sqrt{\alpha}} + \kappa \delta_\eps C_{\theta,1,1} \bigl(1 + \resconst \abs{\omega+1+z}\bigr) + \kappa^2 \resconst \Bigr) \\
			& \qquad + \delta_\eps \Bigl( \frac{\beta}{\sqrt{\alpha}} + C_{\theta,1,1} \bigl(1 + \resconst \abs{\omega+1+z}\bigr) \, \abs{\omega+1+z} \Bigr) \norm[\mathscr{L}(H_\eps)]{R_\eps(z)} \\
			& \eqqcolon \resconst_1 + \delta_\eps c \norm[\mathscr{L}(H_\eps)]{R_\eps(z)}.
	\end{align*}
	Thus, if $\delta_\eps \in (0,\delta_0]$ with $\delta_0 \coloneqq \frac{1}{2c}$,
	then
	\[
		\frac{1}{2} \norm[\mathscr{L}(H_\eps)]{R_\eps(z)} \le (1 - \delta_\eps c) \norm[\mathscr{L}(H_\eps)]{R_\eps(z)} \le \resconst_1,
	\]
	i.e., we have proved the claim with $L \coloneqq 2\resconst_1$.
\end{proof}

Now we come to our main theorem regarding convergence of the spectrum.

\begin{theorem}
  \label{thm:spec_conv}
  Let $A_0$ be an m-sectorial operator with vertex $\omega$,
  semi-angle $\theta$ and associated form $a_0$.  Let $K \subset
  \res{A_0}$ be compact and connected. Then there exist constants
  $\delta_0 > 0$ and $C_{\theta,K}, D_{\theta,K} \ge 0$ with the
  following property: if $a_\eps$ is
  $\delta_\eps$-$\kappa$-quasi-unitarily equivalent to $a_0$ for
  $\delta_\eps \in (0,\delta_0]$, if $(a_\eps)_\eps$ is
  equi-elliptic, and if in addition $K \cap \res{A_\eps} \neq
  \emptyset$, then $K \subset \res{A_\eps}$,
  \begin{equation}
    \label{eq:res_est_sec_pre}
    \norm[\mathscr{L}(H_\eps,H_0)]{\Jdown R(z,A_\eps) - R(z,A_0) \Jdown} \le C_{\theta,K} \delta_\eps
  \end{equation}
  and
  \begin{equation}\label{eq:res_est_sec}
    \norm[\mathscr{L}(H_\eps)]{\Jup R(z,A_0) \Jdown - R(z,A_\eps)} \le D_{\theta,K} \delta_\eps
  \end{equation}
  for all $z \in K$.
\end{theorem}
Note that $C_{\theta,K}, D_{\theta,K} \ge 0$ also depend on $A_0$ and,
as usual, on the constants of equi-ellipticity (see \Def{equi_sec}).
\begin{proof}
  Since $K$ is compact, $K \subset B(0,r)$ for some $r > 0$ and
  \[
  \resconst \coloneqq \sup_{z \in K}
  \norm[\mathscr{L}(H_\eps)]{R_0(z)} < \infty.
  \]
  Choose $\delta_0 = \delta_0(\resconst,r,\omega)$ as in
  Lemma~\ref{lem:res_est_sector}.  Let $\delta_\eps \in (0,\delta_0)$
  and let $a_\eps$ be $\delta_\eps$-$\kappa$-quasi-unitarily
  equivalent to $a_0$. Let $K_0 \coloneqq \res{A_\eps} \cap K$, which
  is non-empty by assumption.  Since $\res{A_\eps}$ is open, the set
  $K_0$ is relatively open in $K$.

  Let $(z_n)$ be a sequence in $K_0$ converging to $z \in K$.  Then
  from Lemma~\ref{lem:res_est_sector} we know that
  $\norm[\mathscr{L}(H_\eps)]{R_\eps(z_n)}$ is bounded, hence $z \in
  \res{A_\eps}$.  We have shown that $K_0$ is closed in $K$. Since $K$
  is connected, $K_0 = K$, i.e., $K \subset \res{A_\eps}$.

	Since by Lemma~\ref{lem:res_est_sector} we have $\norm[\mathscr{L}(H_\eps)]{R_\eps(z)} \le L$
	for some $L > 0$, it follows from~\eqref{eq:Vest} that
	\[
		\norm[\mathscr{L}(H_\eps,H_0)]{\Jdown R_\eps(z) - R_0(z) \Jdown}
			\le \delta_\eps C_{\theta,1,1} \bigl( 1 + \resconst (\abs{\omega} + 1 + r) \bigr) \bigl( 1 + L (\abs{\omega} + 1 + r) \bigr).
	\]
	for all $z \in K$. This is~\eqref{eq:res_est_sec_pre} for
	\[
		C_{\theta,K} \coloneqq C_{\theta,1,1} \bigl( 1 + \resconst (\abs{\omega} + 1 + r) \bigr) \bigl( 1 + L (\abs{\omega} + 1 + r) \bigr)
	\]
	Now~\eqref{eq:res_est_sec} follows from Lemma~\ref{lem:clar} and estimate~\eqref{eq:RHVest}
	with
	\[
		D_{\theta,K} \coloneqq \kappa C_{\theta,K} + \frac{1}{\sqrt{\alpha}} \bigl( 1 + \beta L \bigr).
	\qedhere
	\]
\end{proof}

\begin{remark}
  It can be difficult to check the condition $K \cap \res{A_\eps} \neq
  \emptyset$ of the previous theorem. On the other hand, in the
  classical situation, i.e., if $H_\eps = H_0$ for all $\eps \ge 0$
  and $R(z,A_\eps)$ converges to $R(z,A_0)$ in operator norm, this is
  automatically fulfilled by the fact that the set of invertible
  operators is open in $\mathscr{L}(H_0)$:

  Let $\lambda \in \res{A_0} \cap K$ and $\mu < -\omega$ such that
  $\lambda \neq \mu$.  Then $R_\eps(\mu)$ converges in operator norm
  to $R_0(\mu)$, since $\mu$ is outside the sector $\Sigma_\theta$.
  Moreover, $\lambda \in \res{A_\eps}$ is equivalent with the
  invertibility of $\frac{1}{\mu - \lambda} - R_\eps(\mu)$ by the
  spectral mapping theorem.  For the same reason, $\frac{1}{\mu -
    \lambda} - R_0(\mu)$ is invertible.  Since the set of invertible
  operators is open, $\lambda \in \res{A_\eps}$ for sufficiently small
  $\eps$.
\end{remark}

If the resolvent set is connected, a given compact set $K \subset
\res{A_0}$ can be enlarged to a connected compact set $K' \subset
\res{A_0}$ in such a way that we can guarantee $K' \cap \res{A_\eps}
\neq \emptyset$, so that the theorem applies.  We make this explicit
in the following corollary.  Note that in particular if the spectrum
is real or discrete, the resolvent set is connected.  Hence for
self-adjoint operators and operators with compact resolvent we obtain
spectral convergence.

\begin{corollary}
  \label{cor:conn_spec}
  Assume that $\res{A_0}$ is connected and that $(a_\eps)_{\eps > 0}$
  converges to $a_0$ in the sense of Definition~\ref{def:quasi-uni}.
  Then $\spec{A_\eps}$ converges to $\spec{A_0}$ in the sense of
  Definition~\ref{def:spec_conv}.
\end{corollary}
\begin{proof}
	Let $K \subset \res{A_0}$ be compact. Since $\res{A_0}$ is connected, there exists
	a connected compact set $K' \subset \res{A_0}$ such that $K \subset K'$ and
	$-\omega - 1 \in K'$. In fact, let $R > \abs{\omega} + 1$ be such that $K \subset B(0,R)$,
	and let $(O_{\rho,\mu})_\mu$ denote the family of (open) connected components of the open set
	\[
		O_\rho \coloneqq \bigset{z \in \res{A_0} \cap B(0,R)}
                                        {\dist(z,\spec{A_0}) < \rho}.
	\]
	Then
	\[
		K \cup \{ -\omega - 1 \} \subset \bigcup\nolimits_{\rho, \mu} O_{\rho,\mu},
	\]
	and hence by compactness there exist a finite subcover $(O_{\rho_i,\mu_i})$.
	Let $K''$ the be the union of the
	compact, connected sets $\overline{O}_{\rho_i,\mu_i} \subset \res{A_0}$.
	Now $K''$ has only finitely many connected components. Since $\res{A_0}$
	is arcwise connected, we can join these connected components by finitely
	many paths $(\gamma_k)$ in $\res{A_0}$. Then $K' \coloneqq K'' \cup \bigcup_k \gamma_k$
	is a connected, compact subset of $\res{A_0}$ that contains $K$ and $-\omega-1$.

	Since $-\omega - 1 \in \res{A_\eps}$ for all $\eps \ge 0$ we obtain
	from Theorem~\ref{thm:spec_conv} that $K' \subset \res{A_\eps}$
	if $\delta_\eps$ is sufficiently small. Hence $K \subset \res{A_\eps}$ for small
	$\eps$, which implies the claim.
\end{proof}

In the rest of this section, we show that the discrete spectra of
$(A_\eps)_{\eps>0}$ converge to the discrete spectrum of $A_0$ as the
forms $(a_\eps)_{\eps>0}$ converge to $a_0$. In fact, we show that for
an eigenvalue $\lambda$ of $A_0$ of finite algebraic multiplicity
$m_0(\lambda)$ and for sufficiently small $\delta_\eps$, there exist
exactly $m_0(\lambda)$ eigenvalues of $A_\eps$ near $\lambda$, where
we count the eigenvalues according to their algebraic multiplicity.

Recall that the \emph{algebraic multiplicity} $m_0(\lambda)$ of an isolated
point $\lambda \in \spec{A_0}$ is the rank $\rank P_0 \coloneqq \dim \range P_0$
of the spectral projection
\[
	P_0 \coloneqq \frac{1}{2\pi\im} \int_{\partial B(\lambda,r)} R(z,A_0) \dd z,
\]
where $r > 0$ is such that $\overline{B(\lambda,r)} \cap \spec{A_0} =
\{\lambda\}$, compare Sec.~1.3 of~\cite{ALL01}.  By Cauchy's integral
theorem, this definition does not depend on the particular choice of
$r$, and in fact we could replace the circle $\partial B(\lambda,r)$
by any positively oriented, smooth curve that surrounds $\lambda$, but
no other point of $\spec{A_0}$.

\begin{remark}\label{rem:spec_proj_bounded}
	Since $R(z,A_0)$ is locally bounded as a $\mathscr{L}(H_0,V_0)$-valued function,
	see for example estimate~\eqref{eq:RHVest}, the spectral projection $P_0$ is a bounded operator from $H_0$ to $V_0$.
\end{remark}

\begin{lemma}\label{lem:Jup_inverse}
	There exist $\delta_0 > 0$ such that $\norm[H_\eps]{\Jup f} \ge \frac{1}{2} \norm[H_0]{f}$
	for all $f \in \range P_0$ if $\delta_\eps \in (0,\delta_0]$.
\end{lemma}
\begin{proof}
  For all $f \in V_0$ we have by~\eqref{eq:adjoint_scalar}, \eqref{A3}
  and~\eqref{A4} that
	\begin{align*}
		\normsqr[H_\eps]{f} - \normsqr[H_0]{\Jup f}
			& = \iprod[H_0]{f}{f} - \iprod[H_\eps]{\Jup f}{\Jup f} \\
			& = \iprod[H_\eps]{f - \Jdown \Jup f}{f} + \delta_\eps \norm[H_\eps]{\Jup f} \norm[H_0]{f} \\
			& \le \delta_\eps \norm[V_0]{f} \norm[H_0]{f} + \delta_\eps \kappa \normsqr[H_0]{f}.
	\end{align*}
	Now if $f \in \range P_0$, i.e., $f = P_0f$, and $f \neq 0$, we obtain that
	\begin{align*}
		\norm[H_\eps]{f} - \norm[H_0]{\Jup f}
			& \le \delta_\eps \frac{\norm[\mathscr{L}(H_0,V_0)]{P_0} + \kappa}{\norm[H_\eps]{\Jup f} + \norm[H_0]{f}} \normsqr[H_0]{f} \\
			& \le \delta_0 \bigl( \norm[\mathscr{L}(H_0,V_0)]{P_0} + \kappa \bigr) \norm[H_0]{f}
			= \frac{1}{2} \norm[H_0]{f}
	\end{align*}
	for
	\[
		\delta_0 \coloneqq \frac{1}{2} \bigl( \norm[\mathscr{L}(H_0,V_0)]{P_0} + \kappa \bigr)^{-1}.
	\qedhere
	\]
\end{proof}

We now prove our main theorem about continuous dependence of the discrete spectrum.
For simplicity we assume that $\res{A_0}$ is connected, even though
it would suffice that $\res{A_\eps} \cap B(\lambda,r) \neq \emptyset$
for small $\eps$ and all $r > 0$.

\begin{theorem}
  \label{thm:eigenvalue_convergence}
  Let $\res{A_0}$ be connected, let $\lambda$ be an isolated point of
  $\spec{A_0}$ with finite algebraic multiplicity $m_0(\lambda) \in
  \N$, and let $D$ be a bounded, open set such that $\overline{D} \cap
  \spec{A_0} = \{\lambda\}$.  Then there exists $\delta_0 > 0$ such
  that if $a_\eps$ is $\delta_\eps$-$\kappa$-quasi-unitarily
  equivalent to $a_0$ for $\delta_\eps \in (0,\delta_0]$ and if
  $(a_\eps)_\eps$ is equi-elliptic, then there exist eigenvalues
  $(\lambda_{\eps,i})_{i=1}^{m_0(\lambda)}$ of $A_\eps$ such that
  \[
  \spec{A_\eps} \cap D = \bigl\{ \lambda_{\eps,1}, \ldots,
  \lambda_{\eps,m_0(\lambda)} \bigr\}.
  \]
  Here, the values $(\lambda_{\eps,i})$ are not necessarily pairwise
  different, but rather each value is repeated according to its
  algebraic multiplicity with respect to $A_\eps$.
\end{theorem}
\begin{proof}
	We may assume that $D$ has smooth boundary. In fact, otherwise we can replace $D$
	by an open set $D_1 \subset D$ with smooth boundary still containing $\lambda$.
	Since $(\overline{D} \setminus D_1) \cap \spec{A_\eps} = \emptyset$ for small $\delta_\eps$
	by Corollary~\ref{cor:conn_spec}, the result carries over from $D_1$ to $D$.

	By Corollary~\ref{cor:conn_spec}, the integral
        \begin{equation*}
          P_\eps \coloneqq \frac{1}{2 \pi \im} \int_{\partial D} R_\eps(z) \dd z
        \end{equation*}
        is defined for sufficiently small $\delta_\eps$, and using
        Theorem~\ref{thm:spec_conv} we see that there exist $\delta_1
        > 0$ and $C_1 \ge 0$ such that
	\[
		\norm[\mathscr{L}(H_\eps)]{\Jdown P_\eps - P_0 \Jdown} < C_1 \delta_\eps
	\]
	if $\delta_\eps \in (0,\delta_1]$.

	Now let $u \in \range(P_\eps)$, i.e., $P_\eps u = u$. Then by Lemma~\ref{lem:Jup_inverse}
	there exists $\delta_2 \in (0,\delta_1)$ such that
	\[
		\norm[H_0]{P_0 \Jdown u}
			\ge \norm[H_0]{\Jdown P_\eps u} - \norm[H_0]{(\Jdown P_\eps - P_0 \Jdown) u}
			\ge \frac{1}{2} \norm[H_\eps]{u} - C_1 \delta_\eps \norm[H_\eps]{u}
			> 0
	\]
	whenever $\delta_\eps \in (0,\delta_2]$. This proves that $P_0 \Jdown$ is injective
	on $\range(P_\eps)$, showing that $\rank P_0 \ge \rank P_\eps$ whenever $\delta_\eps \in (0,\delta_2]$.

	For the converse inequality, we interchange the roles of $P_0$
        and $P_\eps$.  In fact, the estimate $\norm[H_0]{\Jdown u} \ge
        \tfrac{1}{2} \norm[H_\eps]{u}$ for $u \in \range(P_\eps)$ can
        be obtained as in \Lem{Jup_inverse} by exploiting the fact
        that \Lem{res_est_sector} and~\eqref{eq:RHVest} provide a
        uniform bound for $\norm[\mathscr{L}(H_\eps,V_\eps)]{P_\eps}$,
        compare also \Rem{spec_proj_bounded}.  Now it readily follows
        that $\rank P_\eps \ge \rank P_0$ for sufficiently small
        $\delta_\eps$.

	Thus there exists $\delta_0 > 0$ such that $m_0(\lambda) =
        \rank P_\eps$ whenever $\delta_\eps \in (0,\delta_0]$.  This
        implies that $\spec{A_\eps} \cap D$ consists of finitely many
        eigenvalues whose algebraic multiplicities add up to
        $m_0(\lambda)$, compare Theorem~1.32 of~\cite{ALL01}.
\end{proof}
\begin{remark}
  Theorem~\ref{thm:eigenvalue_convergence} says that near an isolated
  eigenvalue $\lambda$ of $A_0$ any sufficiently close operator
  $A_\eps$ also possess only isolated eigenvalues, whose
  multiplicities add up to the multiplicity of $\lambda$. This is a
  version of Corollary~A.15 of~\cite{post:06} for non-self-adjoint
  operators, see also~\cite{post:pre11}.
\end{remark}

The following corollary is a trivial consequence of
Theorem~\ref{thm:eigenvalue_convergence} and
Corollary~\ref{cor:conn_spec}.
\begin{corollary}
  Let $\res{A_0}$ be connected, let $\lambda$ be an isolated point of
  $\spec{A_0}$ with finite algebraic multiplicity $m_0(\lambda) \in
  \N$,, and let $(a_\eps)_{\eps > 0}$ converge to $a_0$ in the sense
  of Definition~\ref{def:quasi-uni}. Then the eigenvalues
  $\lambda_{\eps,i}$ in Theorem~\ref{thm:eigenvalue_convergence}
  converge to $\lambda$, i.e., $\lim_{\eps \to 0} \lambda_{\eps,i} \to
  \lambda$ for every $i=1, \ldots, m_0(\lambda)$.
\end{corollary}

\subsection{Invariance and extrapolation}\label{sec:extra}

Assume that $(a_\eps)_{\eps > 0}$ converges to $a_0$ in the sense of
Definition~\ref{def:quasi-uni}.  We have already shown that the
generated semigroups also converge in an appropriate sense, see
Example~\ref{ex:semigroups}.  It is now natural to ask whether certain
properties of the semigroups $(\e^{-tA_\eps})_{\eps > 0}$ are
inherited by $\e^{-tA_0}$ under appropriate assumptions on the
operators $\Jup$ and $\Jdown$.

In this short section, we formulate a simple result of this kind and
apply it to obtain convergence of the semigroups in extrapolation spaces
under natural assumptions.

\begin{theorem}\label{thm:invariance}
  Let $(a_\eps)_{\eps > 0}$ converge to $a_0$ as $\eps \to 0$ in the
  sense of Definition~\ref{def:quasi-uni}, let $\theta \in
  \thetainterval$, and let $\phi \in H^\infty(\Sigma_\theta -
  \omega)$.  For every $\eps \ge 0$, let $C_\eps$ be a closed subset
  of $H_\eps$ such that
\begin{equation}
\label{c0ceps}
		\Jup C_0 \subset C_\eps
		\quad\text{and}\quad
		\Jdown C_\eps \subset C_0.
\end{equation}
	If $\phi(A_\eps)C_\eps \subset C_\eps$ for all $\eps > 0$,
	then $\phi(A_0) C_0 \subset C_0$.
\end{theorem}
\begin{proof}
	By the assumptions,
\begin{equation}
\label{c0c0}
		(\Jdown \phi(A_\eps) \Jup) C_0 \subset C_0
\end{equation}
for all $\eps > 0$.  Thus the result follows from
Corollary~\ref{cor:hol.calc.cor} because the invariance of a closed
set is stable under strong convergence.
\end{proof}

\begin{remark}\label{rem:rescaling}
  In some applications, for example in Section~\ref{sec:robin},
  Condition~\eqref{c0ceps} is only satisfied up to a rescaling of the
  identification operators, i.e., we can write the identification
  operators as
  \[
    \Jup = c_\eps \widetilde{\Jup},\qquad \Jdown = c_\eps^{-1}
    \widetilde{\Jdown}
  \]
  for operators $\widetilde{\Jup}$ and $\widetilde{\Jdown}$ that do
  satisfy~\eqref{c0ceps}.  It is clear that also in this more general
  situation the inclusion~\eqref{c0c0} is satisfied and hence the
  assertion of \Thm{invariance} remains valid.
\end{remark}

It is well-known how invariance of closed convex subsets under the
action of a semigroup on a Hilbert space $H$ generated by an operator
associated with a sesquilinear form can be efficiently characterised
by a Beurling-Deny-type criterion due to Ouhabaz, see Thm.~2.2
of~\cite{Ouh05}.  Assuming that $H = L^2(\Omega,\mu)$ with a measure
space $\Omega$, typical applications of this criterion involve
positivity and $L^\infty$-contractivity (i.e., invariance of the
subset of those $L^2$-functions taking a.e.\ values in the interval
$[0,\infty)$ or $[-1,1]$).

A typical application of \Thm{invariance} is
the following.
\begin{corollary}\label{cor:submarkov}
  Let $p \in [1,\infty]$.  Assume that $H_\eps = L^2(\Omega_\eps)$ for
  measure spaces $\Omega_\eps$, $\eps \ge 0$. Assume that
  $(a_\eps)_{\eps > 0}$ converges to $a_0$, where the operators $\Jup$
  and $\Jdown$ in the definition of convergence are positive
  ($L^p$-contractive).  Assume that the semigroup $(\e^{-tA_\eps})_{t
    \ge 0}$ is positive ($L^p$-contractive) on $H_\eps$ for every
  $\eps > 0$.  Then $(\e^{-tA_0})_{t \ge 0}$ is positive
  ($L^p$-contractive) on $H_0$.
\end{corollary}
\begin{proof}
	Apply Theorem~\ref{thm:invariance} to the closed (and convex) sets
	\[
		C_\eps \coloneqq \bigl\{ u \in L^2(\Omega_\eps) : u \ge 0 \text{ a.e.} \bigr\}
	\]
	and
	\[
		C_\eps \coloneqq \bigl\{ u \in L^2(\Omega_\eps) \cap L^p(\Omega_\eps) : \norm[L^p(\Omega_\eps)]{u} \le 1 \bigr\}
,\qquad p \in [1,\infty],
	\]
	respectively.
\end{proof}

If we are in the situation that the semigroups on $H_\eps =
L^2(\Omega_\eps)$ are $L^\infty$-contractive, we can even establish
convergence in $\mathscr{L}(L^p(\Omega_\eps))$.  We could also state
the result in a more general version for arbitrary interpolation
spaces.  But this would involve several technical assumption that we
prefer to avoid.  It is clear that the analogous result for $1 < p <
2$ holds if we assume the semigroups to be $L^1$-contractive.

\begin{theorem}\label{thm:Lpconv}
  Let $(a_\eps)_{\eps > 0}$ converge to $a_0$ in the sense of
  Definition~\ref{def:quasi-uni} as $\eps \to 0$, assume that $H_\eps
  = L^2(\Omega_\eps)$ for $\eps \ge 0$ with measure spaces
  $(\Omega_\eps)$, let $\theta \in \thetainterval$, let $\phi \in
  H^\infty(\Sigma_\theta - \omega)$, and let $p \in [2,\infty)$.
  Assume that there exists a family $(c_\eps)_{\eps > 0}$ of positive
  real numbers such that $c_\eps \Jup$, $c_\eps^{-1} \Jdown$ and
  $\phi(A_\eps)$ are $L^\infty$-contractive for all $\eps > 0$.
  
  Then $\Jdown \phi(A_\eps) \Jup \to \phi(A_0)$ strongly as operators on
  $L^p(\Omega_0)$.  If $\phi \in H^\infty_{00}(\Sigma_\theta - \omega)$, the
  operators convergence even in operator norm, and in this case we have
  \begin{equation*}
    \norm[\Lin{L^p(\Omega_\eps)}]
         {\Jup \phi(A_0) \Jdown - \phi(A_\eps)} \to 0.
  \end{equation*}
\end{theorem}
\begin{proof}
  By Corollary~\ref{cor:submarkov} and Remark~\ref{rem:rescaling} also
  $\phi(A_0)$ is $L^\infty$-contractive.  Moreover, by
  Corollary~\ref{cor:hol.calc.cor},
  \begin{equation*}
    \norm[H_0]{\phi(A_0)f - \Jdown \phi(A_\eps) \Jup f} \to 0
  \end{equation*}
  for all $f \in H_0$. Now by the interpolation inequality
  \begin{align*}
    & \norm[L^p(\Omega_0)]{\phi(A_0)f - \Jdown \phi(A_\eps) \Jup f} \\
    & \quad \le \norm[L^\infty(\Omega_0)]{\phi(A_0)f - \Jdown
      \phi(A_\eps) \Jup f}^{(p-2)/p} \;\;
      \norm[L^2(\Omega_0)]{\phi(A_0)f - \Jdown \phi(A_\eps) \Jup f}^{2/p} \\
    & \quad \le \bigl( 2 \norm[L^\infty(\Omega_0)]{f} \bigr)^{(p-2)/p} \;\;
         \norm[L^2(\Omega_0)]{\phi(A_0)f - \Jdown \phi(A_\eps) \Jup f}^{2/p}
       \to 0
  \end{align*}
  for all $f$ in the dense subspace $L^2(\Omega_0) \cap
  L^\infty(\Omega_0)$ of $L^p(\Omega_0)$.  Since in addition
	\[
	\norm[\mathscr{L}(L^p(\Omega_0))]{\Jdown \phi(A_\eps) \Jup} \le 1
	\]
	by the Riesz-Thorin interpolation theorem, this proves strong
        convergence.

	Now if $\psi \in H^\infty_{00}(\Sigma_\theta)$, then as in the
        proof of Corollary~\ref{cor:hol.calc.cor} there exists $C_\psi
        \ge 0$ such that
	\[
        \norm[\mathscr{L}(H_0)]{\psi(A_0) - \Jdown \psi(A_\eps) \Jup}
        \le C_\psi \delta_\eps.
	\]
	Moreover, by assumption and Corollary~\ref{cor:submarkov}, see
        also \Rem{rescaling},
	\[
        \norm[L^\infty(\Omega_0)]{\psi(A_0)f - \Jdown \psi(A_\eps)
          \Jup f} \le 2 \norm[L^\infty(\Omega_0)]{f}.
	\]
	Thus by the Riesz-Thorin interpolation theorem
	\[
        \norm[\mathscr{L}(L^p(\Omega_0))]{\psi(A_0) - \Jdown
          \psi(A_\eps) \Jup} \le 2^{(p-2)/p} C_\psi^{p/2}
        \delta_\eps^{p/2} \to 0.
	\]
	Employing \Thm{hol.calc2} instead of \Cor{hol.calc.cor}, the
        same reasoning shows that
        \[
        \norm[\Lin{L^p(\Omega)}]{\Jup \phi(A_0) \Jdown -
          \phi(A_\eps)} \to 0.  \qedhere
        \]
\end{proof}

\section{Simple examples}\label{sec:simple_example}
In this section, we collect some examples to which the theory of the
previous section can be applied without much effort.  On the other
hand, our main application, which involves some delicate calculations,
is contained in a section by its own.

\subsection{Fourier series}\label{sec:fourier}
We start with an almost trivial example.  Let $a_0$ be sectorial form
with form domain $V_0$ on a Hilbert space $H_0$ as introduced in
\Sec{general}, and let $A_0$ be the associated operator.  Assume that
$V_0$ is compactly embedded into $H_0$, i.e., that $A_0$ has compact
resolvent.  For simplicity we also assume that $A_0$ is self-adjoint.

It is classical that in this situation there exists an orthonormal
basis $(e_k)_{k \in \N}$ of $H_0$ consisting of eigenvectors
of $A_0$ to eigenvalues $(\lambda_k)_{k \in \N}$, and
$\lambda_k \to \infty$. We can assume that $\lambda_k \le
\lambda_{k+1}$ for all $k \in \N$, and to make the notation
simpler we assume that $\lambda_1 > 0$.  Passing to an equivalent
norm,
\[
V_0 = \Bigset{f \in H_0} 
      {\normsqr[V_0]{f}
       = \sum_{k=1}^\infty \lambda_k \bigabs{\iprod[H_0]{f}{e_k}}^2 < \infty}.
\]

We explain how this situation can be embedded into our framework.
To this aim, it is convenient to index the Hilbert spaces and operators by
$n \in \N$ instead of $\eps$.
Let $P_n$ denote the orthogonal projection onto
$H_n \coloneqq V_n \coloneqq \operatorname{span}(e_k)_{k=1}^n$,
$\Jnup \coloneqq \Jnup_1 \coloneqq P_n$,
and $\Jndown \coloneqq \Jndown_1 \coloneqq \id$, where $H_n$ and $V_n$
carry the norms induced by $H_0$ and $V_0$, respectively, and let $a_n$
be the restriction of $a_0$ to $V_n$, so that $A_n$ is the restriction of $A_0$
to $H_n = H_n \cap D(A_0)$.

Now~\eqref{A1}, \eqref{A2} and~\eqref{A4} are trivial; in fact, these conditions
hold with $\delta_\eps = 0$ and $\kappa = 1$. Moreover,
\[
	\normsqr[H_0]{f - P_n f}
		= \sum_{k=n+1}^\infty \bigabs{\iprod[H_0]{f}{e_k}}
		\le \frac{1}{\lambda_{n+1}} \normsqr[V_0]{f},
\]
which implies both conditions in~\eqref{A3}. Finally, \eqref{A5} follows from
the fact that
\[
	a_0(f, u) - a_n(P_n f, u)
		= \sum_{k=1}^\infty \lambda_k \iprod[H_0]{f - P_n f}{e_k} \iprod[H_0]{e_k}{u}
		= 0
\]
for all $f \in V_0$ and $u \in V_n$.  Hence the forms $a_n$ converge
to $a_0$ in the sense of Definition~\ref{def:quasi-uni}.

The results in Section~\ref{sec:abstract} now tell us that
\[
	\norm[\mathscr{L}(H_0)]{P_n R(z,A_0) - R(z,A_n)} \to 0,
\]
as well as that other functions of these operators like the generated
semigroup converge in this sense. Convergence of the spectrum as in
Corollary~\ref{cor:conn_spec} and
Theorem~\ref{thm:eigenvalue_convergence} is of course built into this
approximation.

\subsection{Varying coefficients}\label{sec:coeff_conv}
Studying the convergence of elliptic operators with varying
coefficients is a very classical topic. In fact, the underlying spaces
typically do not change, so that the theory in Kato's book applies.
However, sometimes it is convenient to incorporate the coefficients
into the measure of the underlying $L^2$-space.  Although such
problems are still accessible by classical methods if all the norms
are uniformly equivalent, it is quite natural to work with varying
Hilbert spaces instead.

The following example is taken from~\cite{coc-fav:08}, where the
authors proved strong convergence in $\mathrm{C}(\overline{\Omega})$
as well as in $L^p$ for every $p \in [1,\infty)$ for a class of
elliptic operators with Wentzell boundary conditions.  Applying our
results, on the other hand, we obtain convergence in operator norm for
all $p \in (1,\infty)$, see Theorem~\ref{thm:Lpconv}.  Tracing the
constants in the proofs, we in addition have explicit error estimates,
and in particular we know the order of convergence, which answers the
open question that closes~\cite{coc-fav:08}. We also mention the later
article~\cite{coc-gol:08}, where these results are refined by
obtaining a detailed estimate on the order of convergence in operator
norm in $H^1$.

Let $\Omega$ be a bounded Lipschitz domain in $\R^n$.  Then $\Gamma
\coloneqq \partial\Omega$ becomes an oriented compact Riemannian
manifold with Riemannian metric $g$ in a natural way, where the charts
are Lipschitz regular, and the metric is bounded and measurable.  As
in the smooth case, there exists a volume measure $\sigma$ on
$\Gamma$, which coincides with the $(n-1)$-dimensional Hausdorff
measure.  Let $\Sob \Gamma$ be the completion of Lipschitz-continuous
functions $u$ on $\Gamma$ with respect to the norm defined by
\begin{equation*}
  \normsqr[\Sob \Gamma] u
  \coloneqq \int_\Gamma \bigl( \abssqr u + \abssqr[g]{\de u} \bigr) \dd\sigma,
\end{equation*}
where
\begin{equation}
  \label{eq:met.loc}
  \abssqr[g]{\de u}
  = \sum_{i,j=1}^{n-1} g^{ij} \partial_i u \partial_j \conj u
\end{equation}
in a chart $U \subset \Gamma$ with coordinates $\map {x_i} U \R$,
$i=1, \ldots, n-1$, and tangential vectors $\partial_i =
\partial/\partial x_i$.  Moreover, $(g^{ij})$ is the inverse of
$(g_{ij})=(g(\partial_i, \partial_j))$.
For an 
ad hoc definition of Lipschitz-regular manifolds, we refer to~\cite{ABE08}.

Now let the families $(\mathscr{A}_\eps)_{\eps \ge 0} \subset L^\infty(\Omega;\mathscr{L}(\C^n))$,
$(\beta_\eps)_{\eps \ge 0} \subset L^\infty(\Gamma)$,
$(\gamma_\eps)_{\eps \ge 0} \subset L^\infty(\Gamma)$
and $(q_\eps)_{\eps \ge 0} \subset \R$ be bounded in the respective spaces,
and assume that there exist $\alpha > 0$ and $b > 0$
such that for all $\eps \ge 0$ we have $q_\eps \ge \alpha$,
\[
	\Re \iprod[\C^n]{\mathscr{A}_\eps \xi}{\xi} \ge \alpha \abs{\xi}^2
\]
on $\Omega$ for all $\xi \in \C^n$ and $\beta_\eps \ge b$ on $\Gamma$.
For $\eps \ge 0$, define
\[
	H_\eps \coloneqq L^2(\Omega) \times L^2\Bigl(\Gamma;\frac{\dd\sigma}{\beta_\eps}\Bigr)
\]
and
\[
	V_\eps \coloneqq \bigl\{ (u,f) \in H^1(\Omega) \times H^1(\Gamma) : u|_\Gamma = f \bigr\} \subset H_\eps,
\]
and equip these spaces with the natural scalar products.
Note that the space $V_\eps$ and its norm do in fact not depend on $\eps$.

\begin{proposition}
  The family $(a_\eps)_{\eps \ge 0}$ of sesquilinear forms with form
  domains $V_\eps$ which is defined by
  \[
  a_\eps\bigl( (u,u|_\Gamma), (v,v|_\Gamma) \bigr) \coloneqq
  \int_\Omega \iprod[\C^n]{\mathscr{A_\eps} \nabla u}{\nabla v} +
  \int_\Gamma \gamma_\eps u \conj{v} \frac{\dd \sigma}{\beta_\eps} +
  q_\eps \int_\Gamma \iprod[g] {\de u} {\de v} \dd\sigma
  \]
	is equi-elliptic.
\end{proposition}

\begin{proof}
  By the uniformity conditions on the coefficients,
  \[
		\normsqr[H_\eps]{u}
			= \normsqr[L^2(\Omega)]{u} + \normsqr[L^2(\Gamma;\frac{\dd\sigma}{\beta_\eps})]{u}
			\le \normsqr[H^1(\Omega)]{u} + \tfrac{1}{b} \normsqr[H^1(\Gamma)]{u}
			\le \bigl(1 + \tfrac{1}{b}\bigr) \normsqr[V_\eps]{u},
	\]
	which shows that the embedding of $V_\eps$ into $H_\eps$ has
        a uniform constant.  Moreover,
	\begin{align*}
		\abs{a_\eps\bigl( (u,u|_\Gamma), (v,v|_\Gamma) \bigr)}
			& \le \norm[L^\infty(\Omega,\mathscr{L}(\C^n))]{\mathscr{A}_\eps} \norm[L^2(\Omega)]{\nabla u} \norm[L^2(\Omega)]{\nabla v} \\
			& \qquad + \norm[\infty]{\gamma_\eps} \norm[L^2(\Gamma;\frac{\dd\sigma}{\beta_\eps})]{u} \norm[L^2(\Gamma;\frac{\dd\sigma}{\beta_\eps})]{v}
					+ q_\eps \norm[H^1(\Gamma)]{u} \norm[H^1(\Gamma)]{v},
	\end{align*}
	which shows that the forms are uniformly bounded, and
	\begin{align*}
		& \Re a_\eps\bigl( (u,u|_\Gamma), (u,u|_\Gamma) \bigr) \\
			& \qquad \ge \alpha \normsqr[L^2(\Omega)]{\nabla u} + \alpha \normsqr[L^2(\Gamma)]{\de u} - \norm[\infty]{\gamma_\eps} \normsqr[L^2(\Gamma;\frac{\dd\sigma}{\beta_\eps}))]{u} \\
			& \qquad \ge \alpha \Bigl( \normsqr[H^1(\Omega)]{u} + \normsqr[H^1(\Gamma)]{u} \Bigr)
				- \alpha \normsqr[L^2(\Omega)]{u} - \bigl( \norm[\infty]{\gamma_\eps} + \alpha \norm[\infty]{\beta_\eps} \bigr) \normsqr[L^2(\Gamma;\frac{\dd\sigma}{\beta_\eps})]{u},
	\end{align*}
	which shows that the ellipticity constants are uniform with
        respect to $\eps \ge 0$.
\end{proof}

\begin{remark}
  Integration by parts shows that (at least formally) the operator
  $A_\eps$ on $H_\eps$ associated with $a_\eps$ acts as
  \[
  A_\eps \bigl( (u,u|_\Gamma) \bigr) = \Bigl( -\operatorname{div}
  (\mathscr{A_\eps} \nabla u),  \beta_\eps\left((\mathscr{A_\eps} \nabla u) \cdot \nu\right) +
  \gamma_\eps u|_\Gamma - q_\eps \beta_\eps \Delta_\Gamma u|_\Gamma
  \Bigr),
  \]
  where $\nu$ denotes the outer unit normal of $\Omega$, and
  $\Delta_\Gamma$ is the Laplace-Beltrami operator on $\Gamma$, i.e.,
  $A_\eps$ is the operator considered in~\cite{coc-fav:08}.
\end{remark}

\begin{proposition}
  There exists a constant $K$ depending only on $b$ and
  $\norm[\infty]{\beta_0}$ and $\norm[\infty]{\gamma_0}$ such that the
  forms $a_\eps$ are $\delta_\eps$-$\kappa$-quasi-unitarily
  equivalent to $a_0$ for $\kappa = 1$ and
	\[
		\delta_\eps = \Err
                \Bigl(\norm[\infty]{\mathscr{A}_\eps - \mathscr{A}_0}
                + \norm[\infty]{\beta_\eps - \beta_0} +
                \norm[\infty]{\gamma_\eps - \gamma_0} + \abs{q_\eps -
                  q_0} \Bigr).
	\]
	Moreover, this equivalence can be realised by taking the
	identification operators to be the identity operators.
\end{proposition}

\begin{proof}
	Let $\Jup$, $\Jup_1$, $\Jdown$ and $\Jdown_1$ be the identity operators between the respective spaces.
	Then~\eqref{A1}, \eqref{A3} and~\eqref{A4} hold trivially with $\delta_\eps = 0$
	and $\kappa = 1$.
	
	To check~\eqref{A2}, fix $(u,f) \in H_0$ and $(v,g) \in H_\eps$. Then
	\begin{align*}
		& \bigabs{ \iprod[H_\eps]{\Jup (u,f)}{(v,g)} - \iprod[H_0]{(u,f)}{\Jdown (v,g)} } \\
			& \qquad \le \int_\Gamma \bigabs{ \tfrac{1}{\beta_\eps} - \tfrac{1}{\beta_0} } \, \abs{f} \abs{g} \dd\sigma
			\le \norm[\infty]{\beta_\eps - \beta_0} \int_\Gamma \frac{\abs{fg}}{b \sqrt{\beta_0 \beta_\eps}} \dd\sigma \\
			& \qquad \le \frac{\norm[\infty]{\beta_\eps - \beta_0}}{b} \norm[H_0]{(u,f)} \norm[H_\eps]{(v,g)},
	\end{align*}
	i.e., \eqref{A2} holds with $\delta_\eps = b^{-1} \norm[\infty]{\beta_\eps - \beta_0}$,
	compare Remark~\ref{rem:quasirem}.

	Finally, to check~\eqref{A5}, fix $(u,u|_\Gamma) \in V_0$ and $(v,v|_\Gamma) \in V_\eps$. Then
	\begin{align*}
		& \bigabs{ a_0\bigl( (u,u|_\Gamma), \Jdown_1 (v,v|_\Gamma) \bigr) - a_\eps\bigl( \Jup_1 (u,u|_\Gamma), (v,v|_\Gamma) \bigr) } \\
			& \qquad \le \norm[\infty]{\mathscr{A}_\eps - \mathscr{A}_0} \norm[L^2(\Omega)]{\nabla u} \norm[L^2(\Omega)]{\nabla v} \\
				& \qquad\qquad + \Bignorm[\infty]{\frac{\gamma_0}{\beta_0} - \frac{\gamma_\eps}{\beta_\eps}} \norm[L^2(\Gamma)]{u} \norm[L^2(\Gamma)]{v}
				+ \abs{q_0 - q_\eps} \norm[H^1(\Gamma)]{u} \norm[H^1(\Gamma)]{v}.
	\end{align*}
	Finally, note that
	\[
		\Bignorm[\infty]{ \frac{\gamma_0}{\beta_0} - \frac{\gamma_\eps}{\beta_\eps} }
			\le \frac{\norm[\infty]{\beta_\eps \gamma_0 - \beta_0 \gamma_\eps}}{b^2}
			\le \frac{\norm[\infty]{\gamma_0}}{b^2} \norm[\infty]{\beta_\eps - \beta_0}
				+ \frac{\norm[\infty]{\beta_0}}{b^2} \norm[\infty]{\gamma_0 - \gamma_\eps},
	\]
	which concludes the proof.
\end{proof}

It is easy to check using the Beurling-Deny criteria that the
semigroups $(\e^{-tA_\eps})$ are positive and
quasi-contractive in the norm of
\[
	L^\infty(\Omega) \times L^\infty\Bigl( \Gamma; \frac{\dd\sigma}{\beta_\eps} \Bigr),
\]
i.e.,
\[
\norm[L^\infty(\Omega) \times L^\infty( \Gamma;
\frac{\dd\sigma}{\beta_\eps} )]{\e^{-tA_\eps}} \le \e^{r t},
\]
where $r \in \R$ depends only on a lower bound of $(\gamma_\eps)_{\eps
  \ge 0}$.  By duality, we obtain uniform quasi-contractivity also in
the space
\[
	L^1(\Omega) \times L^1\Bigl( \Gamma; \frac{\dd\sigma}{\beta_\eps} \Bigr).
\]
In fact, the adjoint operator satisfies the same
conditions as $A_\eps$ itself.

Thus, by Theorem~\ref{thm:Lpconv} (and its dual version for $p < 2$)
we obtain the following result.  In fact, the proof of
Theorem~\ref{thm:Lpconv} provides an explicit error estimate.
\begin{theorem}
\label{thm:var.coeff}
Assume that $\mathscr{A}_\eps \to \mathscr{A}_0$, $\beta_\eps \to
\beta_0$, $\gamma_\eps \to \gamma_0$ and $q_\eps \to q_0$ uniformly on
$\Omega$ or $\Gamma$, respectively. Then for every $t \ge 0$ and $p
\in (1,\infty)$ the operators $\e^{-tA_\eps}$ converge to $\e^{-tA_0}$
as $\eps \to 0$ in the operator norm of $L^p(\Omega) \times L^p\bigl(
\Gamma; \dd\sigma \bigr)$ and the convergence rate can be estimated by
$\delta_\eps^{p/2}$.
\end{theorem}

\subsection{Degenerate equations in non-divergence form}\label{sec:degenerate}

Now we show that our machinery also applies to the approximation of
degenerate elliptic operators in non-divergence form.  More precisely,
we study the operator $m\Delta$ with Dirichlet boundary conditions on
a bounded domain $\Omega \subset \R^N$ for a possibly
degenerate function $m$. This operator has been studied for example by
Arendt and Chovanec~\cite{AC08}, who investigated under which
conditions its part in $\mathrm{C}_0(\Omega)$ generates a
$\mathrm{C}_0$-semigroup.

Let $m_0\colon \Omega \to (0,\infty)$ be a bounded, measurable
function and assume that $\frac{1}{m_0}\in L^q(\Omega)$, where $q = 1$
if $N=1$, $q > 1$ if $N = 2$, and $q = \frac{N}{2}$ if $N \ge 3$.
Define $m_\eps \coloneqq \max\{ m_0, \eps \}$ Our goal is to show that
the forms associated to the uniformly elliptic operators
$m_\eps\Delta$ converge to the (possibly degenerate) form associated
to the operator $m_0\Delta$ in the sense of our abstract framework as
$\eps \to 0$, where
\begin{multline*}
  \dom(m_\eps\Delta) 
  \coloneqq \bigset{u \in H^1_0(\Omega) 
          \cap L^2(\Omega;\tfrac{\dd x}{m_\eps(x)})}
         {\\ \exists f \in L^2(\Omega;\tfrac{\dd x}{m_\eps(x)})
           \text{ such that } \Delta u = \tfrac{f}{m_\eps}}, 
  \qquad
  (m_\eps\Delta) u  \coloneqq f.
\end{multline*}
Here, in the definition of $D(m_\eps\Delta)$, the expression $\Delta u$ has to
be understood as a distribution.

We start by introducing the forms that give rise to these operators.
Define
\[
	H_\eps \coloneqq L^2(\Omega;\tfrac{\dd x}{m_\eps(x)})
	\quad\text{and}\quad
	V_\eps \coloneqq H^1_0(\Omega).
\]
Note that $V_\eps \subset H_\eps$ even for $\eps = 0$ by the Sobolev
embedding theorem and H\"older's inequality due to the integrability
assumption $\frac{1}{m_0} \in L^q(\Omega)$.  Thus the natural inner
product
\begin{equation}
  \label{eq:nat.iprod}
  \iprod[V_\eps]{u}{v} \coloneqq \int_\Omega \nabla u \conj{\nabla v} +
     \iprod[H_\eps]{u}{v},
\end{equation}
turns $V_\eps \cap H_\eps = V_\eps$ into a Hilbert space.  Here we
used the equivalent norm $u \mapsto \norm{\nabla u}$ on
$H^1_0(\Omega)$. For the norm associated
to~\eqref{eq:nat.iprod}, the embedding constant of $V_\eps$ into
$H_\eps$ is at most $1$. We emphasise that in general the Hilbert
spaces $H_\eps$ do not agree with $H_0$, not even as sets.  We define
the form $a_\eps\colon V_\eps \times V_\eps \to \C$ by
\[
	a_\eps(f,g) \coloneqq \int_\Omega \nabla u \conj{\nabla v}.
\]
Then $a_\eps$ is bounded with constant $M=1$, and $a_\eps$ is elliptic
with constants $\omega = 1$ and $\alpha = 1$.  Hence the family
$(a_\eps)_{\eps \ge 0}$ is equi-elliptic in the sense of
Definition~\ref{def:equi_sec}.  The form $a_\eps$ is associated with
the operator $-m_\eps\Delta$ as defined above, compare~\cite{AC08}.

\begin{theorem}
  \label{thm:deg.eq}
  The forms associated with the operators $m_\eps\Delta$ and
  $m_0\Delta$ are $\delta_\eps$-$\kappa$-quasi unitarily equivalent
  for $\kappa = 1$ and a family $(\delta_\eps)_{\eps > 0}$ of real
  numbers such that $\delta_\eps \to 0$ as $\eps \to 0$.
\end{theorem}
\begin{proof}
	For simplicity, we assume that $N \ge 3$.
	Define
	\[
		\Jdown u \coloneqq \sqrt{\frac{m_0}{m_\eps}}u
		\quad\text{and}\quad
		\Jup f \coloneqq \sqrt{\frac{m_\eps}{m_0}}f.
	\]
	Then $\Jup\colon H_\eps \to H_0$ and $\Jdown\colon H_0 \to H_\eps$
	are isometric isomorphisms, hence unitary. Moreover, $\Jup$
	and $\Jdown$ are inverse to each other, so~\eqref{A2}, \eqref{A3}
	and~\eqref{A4} are satisfied with $\delta_\eps = 0$ and $\kappa = 1$.

	We take $\Jup_1$ and $\Jdown_1$ to be the identity. Then~\eqref{A5} is fulfilled with $\delta_\eps = 0$.
	Moreover, by H\"older's inequality
	\begin{align*}
		\normsqr[H_\eps]{\Jup u - \Jup_1 u}
			& \le \int_\Omega \Bigabs{\sqrt{\frac{m_\eps}{m_0}} - 1}^2 \; \frac{\abs{u}^2}{m_\eps}
			\le \normsqr[L^N(\Omega)]{\tfrac{1}{\sqrt{m_0}} - \tfrac{1}{\sqrt{m_\eps}}} \normsqr[L^{\frac{2N}{N-2}}(\Omega)]{u} \\
			& \le c^2 \delta_\eps^{2/N} \normsqr[V_0]{u}
	\end{align*}
	for all $u \in V_0$ with the embedding constant $c$ of $H^1_0(\Omega)$ into $L^{\frac{2N}{N-2}}(\Omega)$, and for
	\[
		\delta_\eps \coloneqq \int_\Omega \Bigabs{\frac{1}{\sqrt{m_0}} - \frac{1}{\sqrt{m_\eps}}}^N.
	\]
	Since
	\[
		\Bigabs{\frac{1}{\sqrt{m_0}} - \frac{1}{\sqrt{m_\eps}}}^N \le \frac{1}{m_0^{N/2}} \in L^1(\Omega)
	\]
	by assumption and in addition $m_\eps(x) \to m_0(x)$ for all
        $x \in \Omega$, we obtain that $\delta_\eps \to 0$ by the
        dominated convergence theorem.  The other inequality
        in~\eqref{A2} is proved in a similar way; in fact, the
        calculations are symmetric in $m_0$ and $m_\eps$.
\end{proof}

\section{Shrinking tubes with Robin boundary conditions}\label{sec:robin}

In this section we present our main example of convergence of
Laplacians acting in different Hilbert spaces.  We consider a family
of manifolds $X_\eps$ with boundary together with the corresponding
Laplacian $A_\eps$ with (in general) non-local boundary conditions.

We use Robin-type boundary condition of a certain scaling.  In
\Rem{scaling} we compare our approach with the ones used
in~\cite{grieser:pre07} and~\cite{CF08}.

\subsection{The metric graph model}
\label{sec:graph}

As an example of our approximation scheme, we consider a diffusive
process on a family of $(m+1)$-dimensional manifolds $(X,g_\eps)$
converging to a limit space given by a metric graph $X_0$. We will now
present the construction in detail.  We consider compact spaces only.
For the non-compact case, see \Rem{non-compact}.

Let $(\Vx,\Ed,\bd)$ be a directed graph where $\Vx$ and $\Ed$ are
finite sets, the set of \emph{vertices} and \emph{edges}.
Furthermore, $\map \bd \Ed {\Vx \times \Vx}$ encodes the graph
structure and orientation by associating to an edge $\edge \in \Ed$ the
pair $(\bd_-\edge, \bd_+ \edge)$ of its \emph{initial} and \emph{terminal}
vertex.  The orientation is only introduced for convenience.
  The definition of $A_0$ below does not depend on the choice of
  orientation. We
denote by
\begin{equation*}
  \Ed_\vx^\pm := \set {\edge \in \Ed}{\bd_\pm \edge = \vx} \qquadtext{and}
  \Ed_\vx := \Ed_\vx^- \dcup \Ed_\vx^+
\end{equation*}
the set of edges \emph{terminating in $\vx$ $(+)$, starting in $\vx$
  $(-)$} resp.\ \emph{adjacent with $\vx$}.  We denote by $\deg \vx :=
|\Ed_\vx|$ the \emph{degree} of a vertex $v$, i.e., the number of
edges terminating and starting in $\vx$.

Let $X_0$ be the topological graph associated to $(\Vx,\Ed,\bd)$,
i.e., the edges are $1$-dimensional intervals meeting in the vertices
according to the graph structure.  The \emph{metric} structure of
$X_0$ is defined by a function $\map \ell \Ed {(0,\infty)}$
associating to each edge $\edge$ a length $\ell_\edge$.  We parametrise
each edge with a coordinate $s=s_\edge$, i.e., we identify the directed
edge $\edge$ with the associated \emph{metric edge}
$I_\edge:=[0,\ell_\edge]$ in such a way that $\bd_-\edge$ corresponds to
$s=0$ and $\bd_+ \edge$ corresponds to $s=\ell_\edge$.  Introducing the
obvious distance function now turns the topological graph $X_0$ into a
metric space, the \emph{metric graph}.  Similarly, we have a natural
measure on $X_0$ given by the Lebesque measure on each edge $\dd s=\dd
s_\edge$.

The basic Hilbert space is
\begin{equation*}
  \HS_0 := \Lsqr {X_0} := \bigoplus_{\edge \in \Ed} \Lsqr {I_\edge},
\end{equation*}
with norm\footnote{Here and in the sequel, we use the notation
  $\norm[M] f$ for the $\Lsqrsymb$-norm of a measurable function $\map
  u M \C$ on a measure space $M$.} $\normsqr[X_0] f = \sum_{\edge \in
  \Ed}\normsqr[I_\edge]{f_\edge}$, where $\Lsqr {I_\edge}$ is the
usual $\Lsqrsymb$-space with norm given by $\normsqr[I_\edge]
{f_\edge} = \int_0^{\ell_\edge} |f_\edge|^2 \dd s$.

The Hilbert space, which will serve as domain of the sesquilinear form
defined below, is given by
\begin{equation*}
  \HSone_0 :=  \Sob {X_0}
  := \Cont {X_0} \cap \bigoplus_{\edge \in \Ed} \Sob {I_\edge}
\end{equation*}
with norm defined by
\begin{equation*}
  \normsqr[\HSone_0] f
  :=\sum_{\edge\in \Ed} 
        \bigl( \normsqr[I_\edge] {f'} + \normsqr[I_\edge] f \bigr),
\end{equation*}
i.e., a function in $\HSone_0$ is of class $\Sobsymb^1$ along the
edges and also continuous at each vertex.  Trivially, $\norm[\HS_0] f
\le \norm[\HSone_0] f$, i.e., we can choose $C_V=1$
in~\eqref{eq:form_embedding}.

We define the operator governing an evolution process via the
sesquilinear form
\begin{equation}
  \label{eq:qf.graph}
  a_0(f,g)
  :=  a_\Vxnull (f,g) + a_\Ednull(f,g)
\end{equation}
for functions $f,g \in \HSone_0$, where
\begin{align*}
  a_\Vxnull(f,g)
  &:= \sum_{\vx \in \Vx} \sum_{\wx\in \Vx}
         \gamma_{\vx\wx} f(\wx) \conj{g(\vx)} \deg \vx 
    \quad \text{and}\\
  a_\Ednull(f,g)
  &:= \iprod[X_0] {f'}{g'}
  = \sum_{\edge \in \Ed} \int_0^{\ell_\edge} f'_\edge \conj g'_\edge
            \dd s
\end{align*}
for a given coefficient matrix $(\gamma_{\vx\wx})_{\vx,\wx\in \Vx}$.

The following estimate follows easily from a standard Sobolev
  estimate on an interval.  In particular, we have
  \begin{equation}
    \label{eq:eval.vx}
     \normsqr[\Vx]{\ul f} := \sum_{\vx \in \Vx} |f(\vx)|^2 \deg \vx
     \le 4b \normsqr[X_0] {f'} 
       + \frac 8 b \normsqr[X_0] f
  \end{equation}
  for $f \in \Sob{X_0}$, where $\ul f= (f(\vx))_{\vx \in \Vx}$ and $0<
  b \le \min_\edge \ell_\edge$, see e.g.\
  \cite{exner-post:09}.  The next proposition is an
  easy consequence of~\eqref{eq:eval.vx}.
\begin{proposition}
  \label{prp:ell.0}
  The sesquilinear form $a_0$ is well-defined on $\HSone_0
  = \Sob {X_0}$.  Moreover, given $\alpha \in (0,1)$, there exists
  $\omega \ge 0$ such that
  \begin{equation*}
    \Re a_0(f,f) + \omega \normsqr[\HS_0] f \ge \alpha \normsqr[\HSone_0] f
  \end{equation*}
  for all $f \in \HSone_0 = \Sob{X_0}$.  In particular, $a_0$ is
  $\HS_0$-elliptic.
\end{proposition}
It is easily seen that the corresponding operator $A_0$ acts as $(A_0
f)_\edge = -f''_\edge$ on each edge for $f \in \dom A_0$, where $f \in \dom
A_0$ iff $f \in \Cont{X_0} \cap \bigoplus_{\edge \in \Ed}
\Sob[2]{I_\edge}$ and
\begin{equation*}
  \frac 1 {\deg \vx} \sum_{\edge \in \Ed_\vx} f'_\edge(\vx)
    + \sum_{\wx \in \Vx} \gamma_{\vx \wx} f(\wx) = 0,
\end{equation*}
where $f'_\edge(\vx)=-f'_\edge(0)$ if $\vx = \bd_- \edge$ and
$f'_\edge(\vx)=f'_\edge(\ell_\edge)$ if $\vx = \bd_+ \edge$.  Observe that for
a non-diagonal matrix $\gamma$ the vertex conditions defined above
turn out to be non-local.

\subsection{The manifold model}
\label{sec:mfd}
In the sequel, we will construct a manifold $X$ according to the graph
$(\Vx, \Ed, \bd)$ together with a family of metrics $g_\eps$ such
that $(X,g_\eps)$ shrinks to the metric graph $X_0$ in a suitable
sense (see \Fig{mfd}).
\begin{figure}[h]
  \centering
\begin{picture}(0,0)%
\includegraphics[scale=0.7]{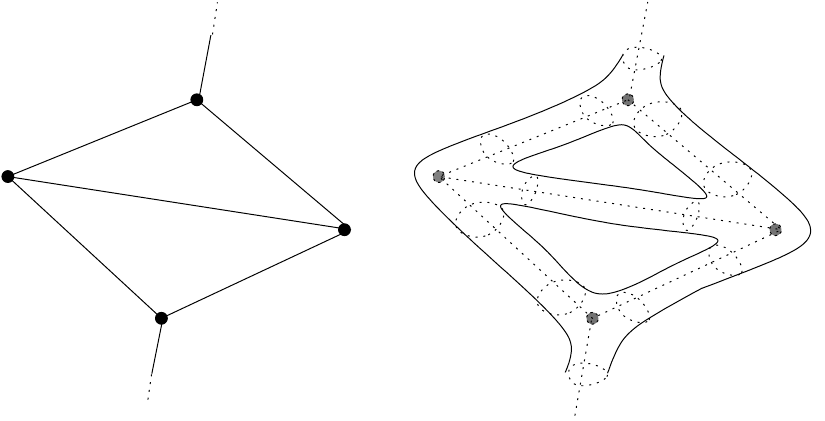}%
\end{picture}%
\setlength{\unitlength}{2901sp}%
%
\begingroup\makeatletter\ifx\SetFigFont\undefined%
\gdef\SetFigFont#1#2#3#4#5{%
  \reset@font\fontsize{#1}{#2pt}%
  \fontfamily{#3}\fontseries{#4}\fontshape{#5}%
  \selectfont}%
\fi\endgroup%
\begin{picture}(6285,3204)(301,-2913)
\put(3601,-2176){\makebox(0,0)[lb]{\smash{{\SetFigFont{12}{14.4}{\rmdefault}{\mddefault}{\updefault}{\color[rgb]{0,0,0}$X_\eps$}%
}}}}
\put(4611,-1581){\makebox(0,0)[lb]{\smash{{\SetFigFont{12}{14.4}{\rmdefault}{\mddefault}{\updefault}{\color[rgb]{0,0,0}$X_\edeps$}%
}}}}
\put(6571,-1546){\makebox(0,0)[lb]{\smash{{\SetFigFont{12}{14.4}{\rmdefault}{\mddefault}{\updefault}{\color[rgb]{0,0,0}$X_\vxeps$}%
}}}}
\put(316,-2176){\makebox(0,0)[lb]{\smash{{\SetFigFont{12}{14.4}{\rmdefault}{\mddefault}{\updefault}{\color[rgb]{0,0,0}$X_0$}%
}}}}
\put(2926,-1681){\makebox(0,0)[lb]{\smash{{\SetFigFont{12}{14.4}{\rmdefault}{\mddefault}{\updefault}{\color[rgb]{0,0,0}$\vx$}%
}}}}
\put(1476,-1451){\makebox(0,0)[lb]{\smash{{\SetFigFont{12}{14.4}{\rmdefault}{\mddefault}{\updefault}{\color[rgb]{0,0,0}$\edge$}%
}}}}
\end{picture}%
  \caption{The metric graph $X_0$ and the family of manifolds
    $(X,g_\eps)$ shrinking to the metric graph.  Here, $(X,g_\eps)$
    can be considered as a subset of $\R^3$, i.e., as a full cylinder
    with boundary consisting of the surface of the pipeline network.}
  \label{fig:mfd}
\end{figure}
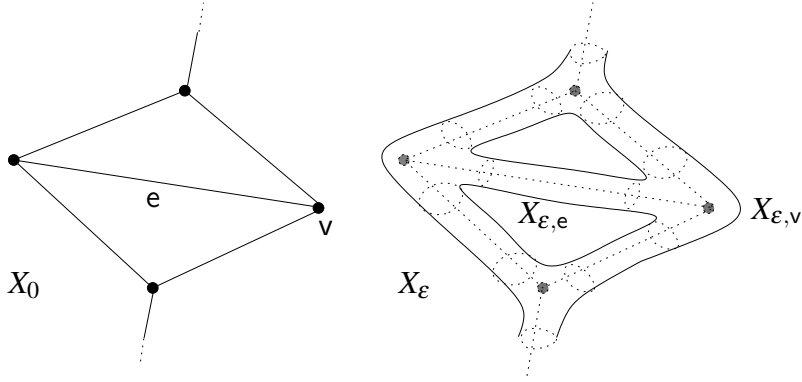

Let $X$ be an $(m+1)$-dimensional connected manifold with boundary $\bd
X$.  We assume that $X$ decomposes as
\begin{equation}
  \label{eq:mfd.dec}
  X = \bigdisjcup_{\vx \in \Vx} X_\vx \cup \bigdisjcup_{\edge \in \Ed} X_\edge,
\end{equation}
where the \emph{vertex} and \emph{edge manifolds}, $X_\vx$ and
$X_\edge$, are compact connected subsets with non-empty interior.
Moreover, we assume that $\{X_\vx\}_{\vx \in \Vx}$ and $\{X_\edge\}_{\edge
  \in \Ed}$ are families of pairwise disjoint sets, respectively
(indicated by $\disjcup$), and
that
\begin{equation*}
  X_\edge \cong I_\edge \times Y_\edge,
\end{equation*}
where $Y_\edge$ is a compact, connected, $m$-dimensional manifold, the
\emph{transversal} or \emph{cross-section manifold} at the edge $\edge$.
Note that $Y_\edge$ has a boundary (as far as $\bd X \cap X_\edge$ is
non-empty).  In the sequel, we will identify $X_\edge$ with the product
$I_\edge \times Y_\edge$.  Finally, we assume that
\begin{equation*}
  Y_{\vxed} := X_\vx \cap X_\edge
  \cong
  \begin{cases}
    Y_\edge, & \text{if $\vx \in \bd \edge$,}\\
    \emptyset & \text{otherwise.}
  \end{cases}
\end{equation*}
Let $g$ be a smooth Riemannian metric on $X$ having product structure
on $X_\edge$, i.e.,
\begin{equation*}
  g_\edge = \dd s_\edge^2 + h_\edge
\end{equation*}
on $X_\edge$, where $h_\edge$ is a Riemannian metric on $Y_\edge$.
Here, and in the sequel, we use the subscripts $(\cdot)_\vx$ and
$(\cdot)_\edge$ to indicate the restriction to $X_\vx$ and $X_\edge$ for
objects on the manifold.  By assumption, $X_\vx$ is a manifold with
boundary in which the disjoint union of transversal manifolds
\begin{equation*}
  Y_\vx = \bigdisjcup_{\edge \in \Ed_\vx} Y_\vxed
\end{equation*}
is embedded.  In addition, the embedding is isometric.
We can think of $X$ as being constructed from the graph $(\Vx, \Ed,
\bd)$ and the family of transversal manifolds $\{Y_\edge\}_{\edge \in
  \Ed}$ and vertex manifolds $\{X_\vx\}_{\vx \in \Vx}$ according to
the graph. 

Let us now define the family of $\eps$-depending metrics on $X$ via
\begin{equation*}
  g_\vxeps := \eps^2 g_\vx \qquadtext{and}
  g_\edeps := \dd s_\edge^2 + \eps^2 h_\edge,
\end{equation*}
i.e., $(X,g_\eps)$ is obtained from the manifold $(X,g)$ by
$\eps$-homothetically shrinking of the vertex manifold $X_\vx$ and the
transversal manifold $Y_\edge$ of the edge manifold $X_\edge$ (thin tube).
Note that $(X,g_\eps)$ defines a smooth Riemannian manifold.  The
smoothness of the metric along the passage from $X_\vx$ to $X_\edge$ is
assured since the original metric $g=g_1$ is assumed to be smooth on
$X$.

If the metric graph $X_0$ is embedded in $\R^{m+1}$, then one can
choose a closed neighbourhood $X_\eps$ of $X_0$ in $\R^{m+1}$ with
smooth boundary and thickness of order $\eps$. The smoothness is
assumed only for simplicity; a Lipschitz boundary would be enough.
Note that a decomposition as in~\eqref{eq:mfd.dec} does not give an
\emph{isometric} decomposition, since the edge neighbourhood
$X_\edeps$ is slightly shorter than $\ell_\edge$ due to the presence
of the vertex neighbourhoods.  Nevertheless, this example can be
treated in the same way after a longitudinal rescaling of the edge
variable.  The error made is only of order $\eps$.  For details, we
refer to Lem.~2.7 of~\cite{exner-post:09} or
\cite[Prop.~5.3.10]{post:pre11}.

The decomposition~\eqref{eq:mfd.dec} induces a decomposition of the
boundary $\Gamma=\bd X$, an $m$-dimensional Riemannian manifold,
\begin{equation*}
  \Gamma = \bigdisjcup_{\vx \in \Vx} \Gamma_\vx \cup
  \bigdisjcup_{\edge \in \Ed} \Gamma_\edge,
\end{equation*}
where $\Gamma_\vx \subset \bd X_\vx$ and $\Gamma_\edge = I_\edge \times
\bd Y_\edge \subset \bd X_\edge$ are pairwise disjoint (or intersect only
in sets of $m$-dimensional measure $0$).
Moreover, we have
\begin{equation*}
  \bd X_\vx = \Gamma_\vx \cup \bigcup_{\edge \in \Ed_\vx} Y_\vxed
     \qquadtext{and}
  \bd X_\edge = \Gamma_\edge \cup \bigcup_{\vx  \in \bd \edge} Y_\vxed.
\end{equation*}

The Riemannian measure associated with a Riemannian manifold $(M,
g_\eps)$ is denoted by $\dd M_\eps$.  In particular, we have
\begin{subequations}
  \label{eq:riem.dens}
  \begin{align}
    \dd X_\vxeps &= \eps^{m+1}\dd X_\vx,&
    \dd \Gamma_\vxeps & = \eps^m \dd \Gamma_\vx,\\
    \dd X_\edeps &= \eps^m \dd X_\edge = \eps^m \dd s_\edge \dd Y_\edge,&
    \dd \Gamma_\edeps &= \eps^{m-1} \dd \Gamma_\edge 
                 = \eps^{m-1} \dd s_\edge \dd \bd Y_\edge.
  \end{align}
\end{subequations}
We will use the abbreviation $X_\eps$, $X_\vxeps$ etc.\ for the measure
spaces $(X,\dd X_\eps)$, $(X_\vx, \dd X_\vxeps)$, etc.

Here and in the sequel, we use the notation
\begin{equation*}
  X_\Ed := \bigcup_{\edge \in \Ed} X_\edge, \qquad
  X_\Vx := \bigcup_{\vx \in \Vx} X_\vx, \qquad
  \Gamma_\Ed := \bigcup_{\edge \in \Ed} \Gamma_\edge
\end{equation*}
etc. for the (disjoint) union of the corresponding manifolds.
Similarly, $X_\Edeps$, $X_\Vxeps$, $\Gamma_\Edeps$ etc.\ denote the
corresponding Riemannian manifolds with the $\eps$-depending metric.

The basic Hilbert space we are working in is $\HS_\eps := \Lsqr
{X_\eps}$.  We often write $\norm[X_\eps] u$ instead of $\norm[\Lsqr
{X_\eps}] u$ for the corresponding norm.
The $\eps$-dependence of the norms for the scaled spaces can easily
calculated using~\eqref{eq:riem.dens}; e.g.~ for $X_\vxeps$ and
$X_\edeps$ we have
\begin{equation*}
  \normsqr[X_\vxeps] u = \eps^{m+1} \int_{X_\vx} \abs u^2 \dd X_\vx
  \qquadtext{and}
  \normsqr[X_\edeps] u = \eps^m
      \int_{I_\edge \times Y_\edge} \abs u^2 \dd Y_\edge \dd x_\edge.
\end{equation*}

Let $\HSone_\eps := \Sob {X_\eps}$ be the Sobolev space of first order
defined as the completion of smooth functions on $X_\eps$ 
with respect to the norm defined by
\begin{equation*}
  \normsqr[\Sob{X_\eps}] u = \normsqr[\Lsqr {X_\eps}] u +
\normsqr[\Lsqr{X_\eps}]{\de u},
\end{equation*}
where $\normsqr[\Lsqr{X_\eps}]{\de u}=\int_X \abssqr[g_\eps] {\de u}
\dd X_\eps$, and $\abssqr[g_\eps]{\de u}$ is given
  in~\eqref{eq:met.loc}.  Trivially, $\norm[\HS_\eps] u \le \norm[\HSone_\eps] u$,
i.e., we can choose $C_V=1$ in~\eqref{eq:form_embedding}.

We define a sesquilinear form by
\begin{equation}
  \label{eq:qf.mfd}
  a_\eps(u,v)
    = \int_X \iprod[g_\eps]{\de u} {\de v} \dd X_\eps
    + \int_\Gamma \beta_\eps u \conj v \dd \Gamma_\eps
    + \int_{\Gamma \times \Gamma} \gamma_\eps
           (u \otimes \conj v) 
         \dd \Gamma_\eps \otimes \dd \Gamma_\eps
\end{equation}
for functions $u \in \HSone_\eps = \Sob{X_\eps}$.  Here $\iprod
[g_\eps] \cdot \cdot$ is the (pointwise) inner product on $T^*X$
defined via the Riemannian metric $g_\eps$.  Moreover, $(u \otimes
\conj v)(x_1,x_2):=u(x_1) \conj{v(x_2)}$.  We assume that $\beta_\eps
\in \Linfty \Gamma$ and $\gamma_\eps \in \Lsqr{\Gamma \times \Gamma}$.
In this case, $a_\eps$ is indeed well-defined for all $u \in
\Sob{X_\eps}$ (see \Prp{ell.eps}).

The associated operator is given by $A_\eps u = -\Delta u = \de^* \de u$
for functions $u \in \dom A_\eps$.  Moreover, $u \in \dom A_\eps$ iff
$u \in \Sob[2]{X_\eps}$ and
\begin{equation*}
  \frac 1 \eps \normder u + \beta_\eps u 
  + \int_\Gamma \gamma_\eps u \dd \Gamma_\eps
  = 0,
\end{equation*}
where the integral is taken with respect to the first variable of
$\map {\gamma_\eps} {\Gamma \times \Gamma} \C$, and where $\normder$
is the normal derivative on $X$.

\begin{remark}
  \label{rem:scaling}
  Let us illustrate the effect of scaling the underlying space for
  Robin boundary conditions in a simple example: Assume that the
  transversal manifold $Y_\edeps$ is isometric to the interval
  $[0,\eps]$ and that there is no non-local contribution, i.e.,
  $\gamma_\eps=0$.  We consider the Laplacian with Neumann boundary
  conditions at $0$ and Robin boundary conditions at $\eps$, i.e.,
  \begin{equation*}
    v'(\eps) + \beta_\eps v(\eps) = 0,
  \end{equation*}
  where $\beta_\eps \in \R \setminus \{0\}$ is the coupling constant.
  An eigenfunction for the Laplacian $\Delta v = -v''$ with Neumann
  boundary conditions at $0$ is of the form
  \begin{equation*}
    v(s)= v(0) \cos (\omega s) \qquadtext{and}
    v(s)= v(0) \cosh (\omega s),
  \end{equation*}
  with eigenvalue $\omega^2$ and $-\omega^2$ if $\beta_\eps>0$ and
  $\beta_\eps<0$, respectively.  The lowest eigenvalue $\mu(\eps)$ is
  then of the same order as $\eps^{-1}\beta_\eps$ for $\eps \to 0$.

  If one chooses \emph{scale-invariant} Robin boundary conditions,
  i.e., $\beta_\eps=\eps^{-1} \beta$ for some $\beta \ne 0$, as e.g.\
  in \cite{grieser:pre07} and~\cite{CF08}, then the lowest eigenvalue
  $\mu(\eps)$ is of order $\eps^{-2}$ and one has to substract the
  divergent term $\mu(\eps)$ in order to expect convergence to a limit.

  Here, we use a different approach.  We assume that the coupling
  constant $\beta_\eps$ is of order $\eps^{3/2}$ on the edge
  neighbourhoods, see~\eqref{eq:ass.beta.gamma.eps'} (actually,
  $\eps^{1+\eta}$ would be enough for some $\eta>0$).  In this case,
  the lowest (transversal) eigenvalue $\mu(\eps)$ is of order
  $\eps^{1/2}$ (resp.\ $\eps^\eta$), and converges to $0$.  We are
  then in the situation, where the Robin Laplacian is close to the
  Neumann Laplacian.  This is the reason why we are in the same
  setting as in the (simpler) Neumann boundary condition case treated
  already in~\cite{post:06}.
\end{remark}

\subsection{Some estimates on the manifold}
\label{sec:est.mfd}

Let us collect some estimates needed later on.  Basically, we need a
Sobolev trace estimate.  Let $M$ be a compact Riemannian manifold of
dimension $n$ with metric $g$ and boundary $\bd M$, and $B$ a compact
$(n-1)$-dimensional submanifold of $\bd M$ carrying the induced
metric.  It follows that there is a constant $\ctr {B,M}>0$ such that
\begin{subequations}
  \begin{equation}
    \label{eq:sob.tr.gen}
    \normsqr[B] u
    \le \ctr {B,M} \bigl( \normsqr[M] {\de u} + \normsqr[M] u \bigr)
  \end{equation}
  for all $u \in \Sob M$.  The constant $\ctr {B,M}$ geometrically
  depends on the shape of $B$ embedded in $M$.

  If we scale the metric by a factor, $g_\eps = \eps^2 g$, then the
  estimate changes to
  \begin{equation}
    \label{eq:sob.tr.gen.eps}
    \normsqr[B_\eps] u
    \le \ctr {B,M} 
    \Bigl(
    \eps \normsqr[M_\eps] {\de u} 
    + \frac 1 \eps \normsqr[M_\eps] u
    \Bigr),
  \end{equation}
\end{subequations}
using $\dd M_\eps = \eps^n \dd M$, $\dd B_\eps = \eps^{n-1} \dd B$ and
the fact that $|\de u|^2_{g_\eps}= \eps^{-2}|\de u|^2_g$.  Here,
$B_\eps$ and $M_\eps$ denote the corresponding Riemannian manifolds
with the $\eps$-depending metric.

We will apply this trace estimate basically in the situations
$(\Gamma_\vx, X_\vx)$, $(Y_\vx, X_\vx)$ and $(\bd Y_\edge, Y_\edge)$.  Let
us first prove the following lemma, which shows that the trace
estimate for $(\bd Y_\edge, Y_\edge)$ gives a trace estimate for the
product $(\Gamma_\edge, X_\edge) = (I_\edge \times \bd Y_\edge, I_\edge \times
Y_\edge)$:
\begin{lemma}
  \label{lem:sob.tr.ed}
  We have
  \begin{equation*}
    \normsqr[\Gamma_\edeps] u
    \le \ctr {\bd Y_\edge, Y_\edge} 
        \Bigl(\eps \normsqr[X_\edeps]{\de_{Y_\edge} u} 
              + \frac 1 \eps \normsqr [X_\edeps] u
        \Bigr)
  \end{equation*}
  for all $u \in \Sob {X_\edeps}$, where $\de_{Y_\edge} u$ denotes the
  exterior derivative with respect to the second variable of $I_\edge
  \times Y_\edge$.
\end{lemma}
\begin{proof}
  Let $u \in \Sob{X_\eps}$ be smooth, then
  \begin{align*}
    \normsqr[\Gamma_\edeps] u
    = \int_0^{\ell_\edge} \normsqr[\bd Y_\edeps] {u(s)} \dd s
    &\le
     \int_0^{\ell_\edge} \ctr  {\bd Y_\edge, Y_\edge}
      \Bigl(
         \eps \normsqr[Y_\edeps]{\de_{Y_\edge} u(s)}
         + \frac 1 \eps \normsqr [Y_\edeps] {u(s)}
      \Bigr)
      \dd s\\
     &= \ctr  {\bd Y_\edge, Y_\edge}
      \Bigl(
        \eps \normsqr[X_\edeps]{\de_{Y_\edge} u}
        + \frac 1 \eps \normsqr [X_\edeps] u
      \Bigr)
  \end{align*}
  using the Sobolev trace estimate~\eqref{eq:sob.tr.gen.eps} for $\bd
  Y_\edge \subset Y_\edge$ pointwise.  Since smooth functions are dense in
  $\Sob{X_\edge}$ and since the operators $\Sob{X_\edge} \to
  \Lsqr{\Gamma_\edge}$, $u \mapsto u\restr {\Gamma_\edge}$ and
  $\Sob{X_\edge} \to \Lsqr{X_\edge,T^*X_\edge}$, $u \mapsto \de_{Y_\edge} u$
  are bounded, the estimate also holds for $u \in \Sob{X_\edge}$.
\end{proof}

It follows from these trace estimates that the global trace operator
$u \mapsto u \restr \Gamma$ is bounded, either as operator $\Sob X \to
\Lsqr \Gamma$ or $\Sob{X_\eps} \to \Lsqr {\Gamma_\eps}$ (see
also \Prp{ell.eps} below).

In the following, we need several averaging operators. 
Let 
\begin{equation}
\label{eq:av.vx}
  \avint_\vx u := \dashint_{X_\vx} u \dd X_\vx.
\end{equation}
denote the average value of $u$ on $X_\vx$ (and also the corresponding
constant function on $X_\vx$).  Here, $\dashint_M := (\volume
M)^{-1}\int$ is the normalised volume integral.
Denote by $\lambda_2(X_\vx)$ the second (first non-vanishing)
eigenvalue of the Neumann problem on $X_\vx$.   Let us now compare the
average of $u$ on $\Gamma_\vx$ with the average of $u$ on $X_\vx$:
\begin{lemma}
  \label{lem:sob.diff2}
  For all $u \in \Sob {X_\vx}$, we have
  \begin{equation*}
    \eps^m \normsqr[\Gamma_\vx]{u - \avint_\vx u}
     \le
    \eps \ctr{\Gamma_\vx, X_\vx}
           \Bigl( \frac 1 {\lambda_2^\Neu(X_\vx)} + 1 \Bigr)
         \normsqr[X_\vxeps] {\de u}.
  \end{equation*}
\end{lemma}
\begin{proof}
  Interpreting $\wt u:= u - \avint_\vx u$ as a function on
  $\Gamma_\vx$ we have
  \begin{equation*}
    \eps^m  \normsqr[\Gamma_\vx] {\wt u}
     \le
    \eps^m  \ctr{\Gamma_\vx, X_\vx} 
          \bigl( \normsqr[X_\vx] {\wt u}
                + \normsqr[X_\vx] {\de u} \bigr)
    \le
    \eps^m \ctr{\Gamma_\vx, X_\vx}
           \Bigl( \frac 1 {\lambda_2^\Neu(X_\vx)} + 1 \Bigr)
         \normsqr[X_\vx] {\de u}
  \end{equation*}
  using Cauchy-Schwarz, the Sobolev trace
  estimate~\eqref{eq:sob.tr.gen} for $\Gamma_\vx \subset X_\vx$ and
  the min-max principle
  \begin{equation*}
    \lambda_2^\Neu(X_\vx) \normsqr[X_\vx] {\wt u}
    \le \normsqr[X_\vx] {\de \wt u}
    =   \normsqr[X_\vx] {\de u},
  \end{equation*}
  since $\wt u$ is orthogonal to the first (constant) eigenfunction of
  the Neumann Laplacian on $X_\vx$.  The scaling property $\eps^{m-1}
  \normsqr[X_\vx] {\de u}= \normsqr[X_\vxeps] {\de u}$ now gives the
  result.
\end{proof}

Next, we compare the average of $u$ on $Y_\vxed$ with the average of
$u$ on $X_\vx$.  To do so, we introduce a partial averaging operator
also needed later on for the identification operators.  For
simplicity, we assume that
\begin{equation}
  \label{eq:p.vol}
  \volume_m Y_\edge = 1
\end{equation}
for all $\edge \in \Ed$.  We set
\begin{equation}
  \label{eq:av.ed}
  (\avint_\edge u)(s) := \dashint_{Y_\edge} u(s,y) \dd Y_\edge(y),
\end{equation}
Note that the integral exists for almost every $s$ and $\avint_\edge u$
defines a function in $\Lsqr{I_\edge}$.  If $s=0$ or $s=\ell_\edge$
denotes the vertex $\vx=\bd_-\edge$ or $\vx=\bd_+ \edge$, respectively, we
also write $(\avint_\edge u)(\vx)$.

The proof of the following lemma is similar to the proof of
\Lem{sob.diff2} (see also Lem.~2.8 of~\cite{exner-post:09}):
\begin{lemma}
  \label{lem:sob.diff}
  We have
  \begin{equation*}
    \eps^m \sum_{\edge \in \Ed_\vx}
        \abs{\avint_\vx u - \avint_\edge u(\vx)}^2
    \le
    \eps \ctr{Y_\vx, X_\vx}
           \Bigl( \frac 1 {\lambda_2^\Neu(X_\vx)} + 1 \Bigr)
         \normsqr[X_\vxeps] {\de u}
  \end{equation*}
  for all $u \in \Sob{X_\vx}$.
\end{lemma}

We finally need an estimate over the vertex neighbourhoods.  It will
assure that in the limit $\eps \to 0$, no family of normalised
eigenfunctions $(u_\eps)_\eps$ with uniformly bounded eigenvalues can
concentrate on $X_\vxeps$, i.e., $\norm[X_\vxeps] u/\norm[X_\eps] u
\to 0$.  A proof of the following estimate was given e.g.\
in~\cite[Lem.~2.9]{exner-post:09}:
\begin{lemma}
  \label{lem:vx.est}
  We have
  \begin{equation*}
    \normsqr[X_\vxeps] u
    \le \eps^2 C_\vx
          \normsqr[X_\vxeps] {\de u}
       + 8 \eps \cvol \vx \Bigl[ b \normsqr[X_\Edepsvx]{u'}
                        + \frac 2 b \normsqr[X_\Edepsvx] u \Bigr]
  \end{equation*}
  for $0 < b \le \min_\edge \ell_\edge$, where 
  \begin{equation}
    \label{eq:const.vx}
     C_\vx := 4 \Bigl[ \frac 1 {\lambda_2(X_\vx)} 
            + \cvol \vx \ctr {Y_\vx, X_\vx} 
                 \bigl( \frac 1 {\lambda_2^\Neu(X_\vx)} + 1 \bigr)
                \Bigr]
     \qquadtext{and}
   \cvol \vx:= \frac{\volume_{m+1} X_\vx} {\deg \vx}.
  \end{equation}
  Moreover, $u'$ denotes the derivative with respect to the
  longitudinal variable $s \in I_\edge$ on each component $X_\edge=I_\edge
  \times Y_\edge$ of $X_{\Ed_\vx}$.
\end{lemma}

\subsection{Equi-ellipticity}
\label{sec:equi-ell}

Let us now show that the family of sesquilinear forms $(a_\eps)_\eps$
is equi-elliptic.  To do so, we need assumptions on $\beta_\eps$ and
$\gamma_\eps$.  We assume that
\begin{subequations}
  \begin{gather}
    \label{eq:ass.gamma}
    \gamma_\eps 
    \in
    \bigl(\Lsqr{\Gamma_\Vxeps} \otimes \Lsqr{\Gamma_\Vxeps}\bigr) \oplus
    \bigl(\Lsqr{\Gamma_\Edeps} \otimes \Lsqr{\Gamma_\Edeps}\bigr)
    \subset \Lsqr{\Gamma_\eps} \otimes \Lsqr{\Gamma_\eps},\\
    \label{eq:ass.beta.gamma.eps}
    \norm[\infty] {\beta_\Vxeps} + \norm[\Gamma_\Vxeps \times
    \Gamma_\Vxeps] {\gamma_\eps} \le C_{\beta,\gamma,\Vx}, \quad
    \norm[\infty] {\beta_\Edeps} + \norm[\Gamma_\Edeps
    \times \Gamma_\Edeps] {\gamma_\eps} \le \eps C_{\beta,\gamma,\Ed}
  \end{gather}
\end{subequations}
for all $\eps>0$ small enough, where $\beta_{\Vxeps}$ is the
restriction of $\beta_\eps$ to $\Gamma_\Vx$ etc.  Note that we assumed for
simplicity that $\gamma_\eps$ only couples edge neighbourhoods with
edge neighbourhoods and vertex neighbourhoods with vertex
neighbourhoods.

\begin{proposition}
  \label{prp:ell.eps}
  Assume that~\eqref{eq:ass.gamma}--\eqref{eq:ass.beta.gamma.eps}
  are fulfilled. Then, $a_\eps(u,u)$ is well-defined for $u \in
  \HSone_\eps = \Sob{X_\eps}$.  Moreover, given $\alpha \in (0,1)$,
  there exists $\omega \ge 0$ and $\eps_0=\eps_0(\alpha)>0$ such that
  \begin{equation*}
    \Re a_\eps(u,u) + \omega \normsqr[\HS_\eps] u
    \ge \alpha \normsqr[\HSone_\eps] u
  \end{equation*}
  for all $u \in \HSone_\eps$ and all $\eps \in (0,\eps_0]$.  In
  particular, $(a_\eps)_{\eps \in (0,\eps_0]}$ is an $(\HS_\eps)_{\eps
    \in (0,\eps_0]}$-equi-elliptic family and $\HS_\eps=\Lsqr
  {X_\eps}$.
\end{proposition}
\begin{proof}
Let us show that~\eqref{eq:form_coercive} holds with uniform
  constants $\omega$ and $\alpha$.  Estimate~\eqref{eq:form_bdd} can
  be seen similarly; and~\eqref{eq:form_embedding} is fulfilled with
  $c_V=1$.

  We start estimating the difference $a_\eps(u,u)-\normsqr[X_\eps]{\de
    u}$.  We have
  \begin{equation*}
    \bigl| a_\eps(u,u) - \normsqr[X_\eps]{\de u} \bigr|
    \le 
      C_{\beta,\gamma,\Vx} \normsqr[\Gamma_\Vxeps] u
      + \eps  C_{\beta,\gamma,\Ed} \normsqr[\Gamma_\Edeps] u
  \end{equation*}
  by Cauchy-Schwarz, Fubini,
  \eqref{eq:ass.gamma}--\eqref{eq:ass.beta.gamma.eps}.  It follows
  from the Sobolev trace estimate~\eqref{eq:sob.tr.gen.eps} and
  \Lem{vx.est} that
  \begin{align*}
    \normsqr[\Gamma_\Vxeps] u
    &\le \max_\vx \ctr {\Gamma_\vx, X_\vx}
              \Bigl(\eps \normsqr[X_\Vxeps] {\de u}
                   + \frac 1 \eps \normsqr[X_\Vxeps] u
              \Bigr)\\
    &\le \max_\vx (\ctr {\Gamma_\vx, X_\vx} + C_\vx) 
                 \eps \normsqr[X_\Vxeps] {\de u}
       + 16 \max_\vx \ctr {\Gamma_\vx, X_\vx}
                  \cvol \vx \Bigl[ b \normsqr[X_\Edeps]{u'}
                        + \frac 2 b \normsqr[X_\Edeps] u \Bigr]
  \end{align*}
  for $0 < b \le \min_\edge \ell_\edge$.  For $\normsqr[\Gamma_\Edeps] u$,
  we have the estimate
  \begin{equation*}
    \normsqr[\Gamma_\Edeps] u
    \le \max_\edge \ctr {\bd Y_\edge, Y_\edge}
              \Bigl(\eps \normsqr[X_\Edeps] {\de u}
                   + \frac 1 \eps \normsqr[X_\Edeps] u
              \Bigr)
  \end{equation*}
  by \Lem{sob.tr.ed}.  It follows that
  \begin{equation*}
    \bigl| a_\eps(u,u) - \normsqr[X_\eps]{\de u} \bigr|
    \le C(\eps,b) \normsqr[X_\eps] {\de u}
      + \omega(a) \normsqr[X_\eps] u,
  \end{equation*}
  where
  \begin{align*}
    C(\eps,b) &:= \max_{\vx,\edge} \bigl\{\eps C_{\beta,\gamma,\Vx}
                       (\ctr {\Gamma_\vx, X_\vx} + C_\vx), \;
                      16 b C_{\beta,\gamma,\Vx} 
                               \ctr {\Gamma_\vx, X_\vx} \cvol \vx, \;
                      \eps^2 C_{\beta,\gamma,\Ed}
                                 \ctr {\Gamma_\edge, X_\edge}\bigr\},\\
    \omega(b)  &:= \max_{\vx,\edge} \bigl\{16 b^{-1} 
                       C_{\beta,\gamma,\Vx} 
                             \ctr {\Gamma_\vx, X_\vx} \cvol \vx, \;
                      C_{\beta,\gamma,\Ed} \ctr {\Gamma_\edge, X_\edge}
               \bigr\}.
  \end{align*}
  For $\alpha \in (0,1)$ we choose
  \begin{equation*}
    \eps_0 := \min_{\vx,\edge} \Bigl\{1, \;
                 \frac{1-\alpha}
                      {C_{\beta,\gamma,\vx}
                           (\ctr {\Gamma_\vx, X_\vx} + C_\vx)},\;
                 \frac{1-\alpha}
                      {C_{\beta,\gamma,\edge} \ctr {\Gamma_\edge, X_\edge}}
              \Bigr\}
  \end{equation*}
  and
  \begin{equation*}
    b := \min_{\vx, \edge} \Bigl\{\ell_\edge, \;
                     \frac{1-\alpha}
                     {16 C_{\beta,\gamma,\Vx} 
                               \ctr {\Gamma_\vx, X_\vx} \cvol \vx}
              \Bigr\}.
  \end{equation*}
  Then $C(\eps,b) \le C(\eps_0,b) \le 1-\alpha$ and we have
  \begin{align*}
    \Re a_\eps(u)
    &\ge \normsqr[X_\eps] {\de u}
     - \bigl| a_\eps(u) - \normsqr[X_\eps]{\de u} \bigr|\\
    &\ge \bigl(1-C(\eps_0,b)\bigr) \normsqr[X_\eps] {\de u}
     - \omega(b) \normsqr[X_\eps] u
    \ge \alpha \normsqr[X_\eps] {\de u}
     - \omega \normsqr[X_\eps] u
  \end{align*}
  for all $\eps \in (0,\eps_0]$ with $\omega := \omega(b)$.  In
  particular, the family $(a_\eps)_\eps$ is equi-elliptic.
\end{proof}

\subsection{The identification operators}
\label{sec:id.ops}

We now fix the identification operators $\Jup$ and $\Jdown$ similar as
in~\cite{post:06} (see also \Rem{scaling}).  In particular, we set
\begin{equation}
  \label{eq:j.up}
  \map \Jup {\Lsqr X}{\Lsqr {X_\eps}}, \qquad
  (\Jup f)_\vx := 0, \quad (\Jup f)_\edge := f_\edge \otimes \1_\edeps,
\end{equation}
where we use the decomposition of $u=\Jup f$ with respect
to~\eqref{eq:mfd.dec}. Here $\1_\edeps(y):=\eps^{-m/2}$ for all $y \in
Y_\edge$, thus
\begin{equation*}
  (\Jup f)_\edge(s,y)=\eps^{-m/2} f_\edge(s).
\end{equation*}
Note that $\norm[H_\eps] {\Jup f} \le \norm[H_0] f$.
For $\Jdown$, we just choose the adjoint, i.e., $\Jdown := (\Jup)^*$.
In particular,~\eqref{A2} is fulfilled and we have
\begin{equation*}
  (\Jdown u)_\edge = \eps^{m/2} \avint_\edge u.
\end{equation*}
Moreover, we need the corresponding identification operators on the
level of quadratic form domains. As in~\cite{post:06}, we define
\begin{equation*}
  (\Jup_1 f)_\edge := (\Jup f)_\edge, \qquad
  (\Jup_1 f)_\vx := \eps^{-m/2}f(\vx)
\end{equation*}
(see~\eqref{eq:av.ed} for the notation).  Note that $f(\vx)$ is well
defined for $f \in \HSone_0$, and that $\Jup_1 f \in \HSone_\eps$. For
the operator in the opposite direction, we choose
\begin{align*}
  (\Jdown_1 u)_\edge (s)
  :=& (\Jdown u)_\edge(s)
      + \eps^{m/2} \sum_{\vx \in \bd \edge}
           \chi_{\vxed}(s)\bigl(\avint_\vx u - \avint_\edge u(\vx)\bigr)\\
  =& \eps^{m/2}\Bigl( \avint_\edge u(s) + \sum_{\vx \in \bd \edge}
           \chi_\vxed(s)\bigl(\avint_\vx u - \avint_\edge u(\vx)\bigr)\Bigr)
\end{align*}
where $\chi_\vxed$ is the continuous function on the metric edge
$I_\edge$ with $\chi_\vxed(\vx)=1$, $\chi_\vxed$ being affine linear on
$I_\vxed:=\set{s \in I_\edge}{d(s,v) \le \ell_0}$ and $\chi_\vxed(s)=0$
for $s \in I_\vxed$. Recall the definition of $\avint_\vx u$
in~\eqref{eq:av.vx}.  In particular, it is easy to see that $(\Jdown_1
u)_\edge(\vx)=\eps^{m/2}\avint_\vx u$, independently of the edge $\edge
\in \Ed_\vx$, i.e., $\Jdown_1 u \in \mc V$.

Let $\gamma$ be the matrix of \Sec{graph}.  We additionally need that
\begin{equation}
  \label{eq:ass.beta.gamma.eps'}
  \norm[\Lin{\lsqr \Vx}] {\wt \gamma_\eps - \gamma}
  \le \eps^{1/2} C'_{\beta,\gamma,\Vx}
         \quadtext{and}
  \norm[\infty] {\beta_\Edeps} + \norm[\Gamma_\Edeps \times
                    \Gamma_\Edeps] {\gamma_\eps}
  \le \eps^{3/2} C'_{\beta,\gamma,\Ed},
\end{equation}
where $\tilde{\gamma}_\eps$ is the $|\Vx| \times | \Vx|$-matrix
defined by
\begin{equation*}
  \wt \gamma_\vxwxeps
  := \frac 1 {\deg \vx}
  \Bigl(
  \delta_\vxwx \int_{\Gamma_\vx} \beta_\eps \dd \Gamma_\vx 
  + \eps^m \int_{\Gamma_\vx \times \Gamma_\wx} 
  \gamma_\eps \dd \Gamma_\vx \otimes \dd \Gamma_\wx
  \Bigr).
\end{equation*}
Moreover, $(\wt \gamma_\eps \phi)(\vx) := \sum_{\wx \in \Vx} \wt
\gamma_\vxwxeps \phi(\wx)$ denotes the corresponding operator in the
Hilbert space $\lsqr \Vx$ with weighted norm $\normsqr[\Vx] \phi :=
\sum_\vx \abssqr{\phi(\vx)} \deg \vx$.  Note that the second
condition in~\eqref{eq:ass.beta.gamma.eps'} is stronger than the
second condition in~\eqref{eq:ass.beta.gamma.eps}.

\begin{proposition}
  \label{prp:a5}
  Assume that~\eqref{eq:ass.gamma}--\eqref{eq:ass.beta.gamma.eps}
  and~\eqref{eq:ass.beta.gamma.eps'} hold, then we have
  \begin{equation*}
    \bigabs{a_\eps \bigl(\Jup_1 f, u)
               - a_0 (f, \Jdown_1 u\bigr)}
    \le \delta_\eps \norm[\HSone_0] f \norm[\HSone_\eps] u
  \end{equation*}
  for all $f \in \HSone_0=\Sob {X_0}$, $u \in \HSone_\eps =
  \Sob{X_\eps}$ and $\eps \in (0,\eps_0]$, where
  $\delta_\eps=\Err(\eps^{1/2})$ depends only on the geometry of the
  unscaled manifold $X$ and the metric graph.
\end{proposition}
\begin{proof}
  In order to verify the estimate, we will split the
  estimate in its vertex and edge part.  For the edge contribution, we
  have
  \begin{multline*}
    \bigabs{a_\Edeps \bigl(\Jup_1 f, u)
               - a_\Ednull (f, \Jdown_1 u\bigr)}\\
    =  \eps^{m/2}\Bigl| 
       \eps^{-1} \Bigl(\int_{\Gamma_\Ed} \beta_\eps f \conj u \dd \Gamma
                + \eps^{m-1} 
                  \int_{\Gamma_\Ed \times \Gamma_\Ed}
                    \gamma_\eps (f \otimes \conj u) 
                  \dd \Gamma \otimes \dd \Gamma
                 \Bigr)\\
                + \sum_{\edge \in \Ed} \sum_{\vx \in \bd \edge} \int_{I_\edge}
                \chi_\vxed' f'_\edge
                       (\avint_\vx \conj u - 
                                \avint_\edge \conj u(\vx)) \dd s 
    \Bigr|.
  \end{multline*}
  The first two integrals can be estimated by
  \begin{multline*}
    \eps C'_{\beta,\gamma,\Ed} (\cvol \Ed)^{1/2}
           \norm[X_0] f \norm[\Gamma_\Edeps] u\\
    \le \eps^{1/2} C'_{\beta,\gamma,\Ed}  (\cvol \Ed)^{1/2}
    \norm[X_0] f \bigl(\max_\edge \ctr {\bd Y_\edge ,Y_\edge} 
                      (\eps^2 \normsqr[X_\Edeps] {\de u}
                              + \normsqr[X_\Edeps] u) \bigr)^{1/2}\\
    \le \eps^{1/2} C'_{\beta,\gamma,\Ed} 
                (\cvol \Ed \max_\edge \ctr {\bd Y_\edge ,Y_\edge})^{1/2}
           \norm[X_0] f \norm[\Sob{X_\eps}] u
  \end{multline*}
  using the assumption in~\eqref{eq:ass.beta.gamma.eps'},
  Cauchy-Schwarz, \Lem{sob.tr.ed} and $\eps \le 1$.  Here, we have set
  $\cvol \Ed := \max_\edge (\volume_{m-1} \bd Y_\edge)$.  The last
  term of the edge contribution is small since
  \begin{multline*}
   \eps^{m/2}\Bigl| 
                \sum_{\edge \in \Ed} \sum_{\vx \in \bd \edge} \int_{I_\edge}
                \chi_\vxed' f'_\edge
                       (\avint_\vx \conj u - 
                                \avint_\edge \conj u(\vx)) \dd s 
             \Bigr|\\
    \le  2\eps^{m/2} \ell_0^{-1/2} \norm[X_0] {f'}
      \Bigl(\sum_{\vx \in \Vx} \sum_{\edge \in \Ed_\vx} 
                 \bigl|\avint_\vx u - \avint_\edge
                 u(\vx) \bigr|^2 \Bigr)^{1/2}\\
    \le 2 \eps^{1/2}\max_\vx \biggl(\frac {\ctr {Y_\vx, X_\vx}} {\ell_0} 
                             \biggr)^{1/2} 
         \Bigl(\frac 1{\lambda_2^\Neu(X_\vx)} + 1\Bigr)^{1/2}
            \norm[X_0] {f'} \norm[X_\Vxeps]{\de u}
  \end{multline*}
  using Cauchy-Schwarz again, the fact that $\normsqr[I_\edge]
  {\chi'_\vxed} = 1/\ell_0$, where $\ell_0 = \min_\edge \{\ell_\edge,1\}$,
  and \Lem{sob.diff}.

  For the vertex contribution, we have
  \begin{multline}
    \label{eq:vx.contr}
   \bigabs{a_\Vxeps \bigl(\Jup_1 f, u)
               - a_\Vxnull (f, \Jdown_1 u\bigr)}
    =  \eps^{m/2}\Bigl|
       \sum_{\vx \in \Vx}  f(\vx) \int_{\Gamma_\vx} \beta_\eps  \conj u
                                \dd \Gamma_\vx \\
         + \sum_{\vx, \wx \in \Vx} f(\wx)
              \Bigl(  \eps^m \int_{\Gamma_\vx \times \Gamma_\wx} 
                         \gamma_\eps (\1 \otimes \conj u) 
                  \dd \Gamma_\vx \otimes \dd \Gamma_\wx
              - \gamma_{\vxwx}  (\deg \vx) \avint_\vx \conj u \Bigr)
    \Bigr|\\
   \le \eps^{m/2}
      \Bigabs{ \sum_{\vx,\wx \in \Vx}  
                 f(\wx) (\wt \gamma_\vxwxeps - \gamma_\vxwx)
        \avint_\vx \conj u) \deg \vx} \hspace*{20ex}\\
      + \eps^{m/2}
         \sum_{\vx \in \Vx}
           \Bigr(\abs{f(\vx)}\norm[\Gamma_\vx]{\beta_\eps}
                 + \sum_{\wx \in \Vx} \abs{f(\wx)}
               \norm[\Gamma_\vxeps \times \Gamma_\wxeps] {\gamma_\eps}
               \norm[\Gamma_\vx] \1 \Bigr)
         \norm[\Gamma_\vx]{u - \avint_\vx u}
  \end{multline}
  since the derivative vanishes as $\Jup_1 f$ is constant on $X_\vx$,
  and where we  replaced $u$ by $\avint_\vx u + (u - \avint_\vx u)$ in the
  first two integrals.
  The first sum of the last estimate can be estimated by
  \begin{equation*}
    \eps^{m/2} \bigabs{\iprod[\Vx] {\ul f}
           {(\wt \gamma_\eps - \gamma) \ul u}}
    \le \norm[\Lin{\lsqr \Vx}]{\wt \gamma_\eps - \gamma} 
        \norm[\Vx] {\ul f} (\eps^{m/2} \norm[\Vx]{\ul u}),
  \end{equation*}
  where $\ul f= (f(\vx))_{\vx \in \Vx}$ and $\ul u = (\avint_\vx
  u)_{\vx \in \Vx}$.  Now,
  \begin{multline*}
    \eps^m\normsqr[\Vx] {\ul u}
    \le \eps^m\sum_{\vx \in \Vx} \frac {\deg \vx}{\volume_{m+1}{X_\vx}}
         \normsqr[X_\vx] u\\
    \le \cvol \Vx' \cdot
      \Bigl(\eps \max_\vx C_\vx
          \normsqr[X_\Vxeps] {\de u}
          + \frac {16 \max_\vx \cvol \vx}{\ell_0} 
             \bigl(\normsqr[X_\Edeps]{u'} + \normsqr[X_\Edeps] u
             \bigr)
       \Bigr)
  \end{multline*}
  by \Lem{vx.est}, where $\cvol \Vx' := \max_\vx (\deg
  \vx)/(\volume_{m+1} X_\vx)$.  In particular, the first sum equals
  $\eps^{m/2} \bigabs{\iprod[\Vx] {\ul f} {(\wt \gamma_\eps - \gamma)
      \ul u}}$ and can be estimated from above by
  \begin{equation*}
    \eps^{1/2} C'_{\beta,\gamma,\Vx}
       \biggl(
         \frac{8 \cvol \Vx' }{\ell_0}
            \max_\vx \Bigl\{ C_\vx, \;
                    \frac{16 \cvol \vx}{\ell_0}
                 \Bigr\}
       \biggr)^{1/2}
     \norm[\Sob{X_0}] f  \norm[\Sob{X_\eps}] u
  \end{equation*}
  by~\eqref{eq:eval.vx} and since $\eps \le 1$.  The second summand of
  the right hand side of~\eqref{eq:vx.contr} can be estimated by
  \begin{multline*}
    \eps^{m/2}
        \bigl(
           \norm[\infty] {\beta_\eps}
           + \norm[\Gamma_\Vxeps \times \Gamma_\Vxeps] {\gamma_\eps}
        \bigr)
        \norm[\Vx]{\ul f}
       \Bigl( \sum_{\vx \in \Vx} 
         \frac{\volume \Gamma_\vx}{\deg \vx} 
            \normsqr[\Gamma_\vx]{u - \avint_\vx u}
       \Bigr)^{1/2}\\
    \le \eps^{1/2}
        C_{\beta,\gamma,\Vx}
       \max_{\vx,\edge} \biggl(
          \frac{8 \cvol \vx'' \ctr {\Gamma_\vx, X_\vx}} {\ell_\edge}
          \Bigl(
            \frac 1 {\lambda_2^\Neu(X_\vx)} + 1
          \Bigr)
       \biggr)^{1/2}
     \norm[\Sob{X_0}] f  \norm[X_\Vxeps] {\de u}
  \end{multline*}
  using \eqref{eq:ass.beta.gamma.eps}, \eqref{eq:eval.vx} and
  \Lem{sob.diff2}, where $\cvol \vx'' := (\volume_m \Gamma_\vx)/(\deg
  \vx)$.
\end{proof}

Let us now formulate the main theorem of this section. 
\begin{theorem}
  \label{thm:conv.mfd}
  Assume that~\eqref{eq:ass.gamma}--\eqref{eq:ass.beta.gamma.eps}
  and~\eqref{eq:ass.beta.gamma.eps'} are fulfilled.  Then the
  sesquilinear forms $(a_\eps)_{\eps \in [0,\eps_0]}$ form an
  equi-elliptic family for some $\eps_0 > 0$.  Moreover, $a_\eps$ is
  $\delta_\eps$-$\kappa$-quasi-unitarily equivalent to $a_0$ for
  $\delta_\eps = \Err(\eps^{1/2})$ and $\kappa=1$.

  In particular, the convergence results of \Sec{general} apply, e.g.,
  the spectra $\spec {A_\eps}$ of the associated operators $A_\eps$
  converges to the spectrum $\spec{A_0}$ of $A_0$ in the sense of
  \Def{spec_conv}.
\end{theorem}
\begin{proof}
  Condition~\eqref{A5} has been shown in \Prp{a5}.  The other
  conditions have already been shown in~\cite{post:06} or
  \cite[Prp.~3.2]{exner-post:09}.  Note that the spectrum of $A_\eps$
  and $A_0$ is purely discrete, since the underlying spaces are
  compact.
\end{proof}

\begin{remark}
  \label{rem:non-compact}
  If the graph $X_0$ and the corresponding manifold $X_\eps$ are not
  compact, the corresponding forms $a_0$ and $a_\eps$ are still
  (equi-)sectorial and fulfil \Def{quasi-uni} provided we have a
  uniform control of the geometry of the graph and the manifold
  building blocks (see the constants in the proofs).  For example, we
  need a positive lower bound on the edge length, i.e., $\inf_\edge
  \ell_\edge>0$ and a uniform finite upper bound on the Sobolev trace
  constants like $\sup_\vx \ctr{\Gamma_\vx,X_\vx}<\infty$.  The
  uniform control of the geometry is in particular fulfilled if there
  is a finite set of of manifolds $\mc M$ such that the building
  blocks $X_\vx$ and $Y_\edge$ of the manifold $X$, constructed according
  to the graph $(\Vx,\Ed,\bd)$, are isometric to a member in $\mc M$.
  Coverings of compact spaces provide such examples.


\end{remark}

\begin{remark}
  \label{rem:pos.contr.mfd}
  Under suitable conditions on the coefficients we can apply
  \Thm{Lpconv} in the context of the approximation results of this
  section.  More precisely, assume e.g.\ that $\beta_\eps \ge 0$ and
  $\gamma_\eps = 0$. Then it can easily be verified that the forms
  $a_\eps$ satisfy the Beurling-Deny conditions for all $\eps\ge 0$.
  Thus the associated semigroups are positive and
  $L^\infty$-contractive. Thus for $\phi(z)=\e^{-tz}$ the assumptions
  of \Thm{Lpconv} are satisfied with $c_\eps = \eps^{m/2}$. Hence
  \begin{equation*}
      \norm[\Lin{L^p(X_\eps)}]{\Jup \e^{-tA_0} \Jdown -  \e^{tA_\eps}}
      \to 0
  \end{equation*}
  as $\eps \to 0$.
\end{remark}

\textit{Bibliography.}


\newcommand{\etalchar}[1]{$^{#1}$}
\providecommand{\bysame}{\leavevmode\hbox to3em{\hrulefill}\thinspace}
\def\cprime{$'$}

\end{document}